\documentclass{amsart}
\usepackage{amsmath}
\usepackage{amssymb}
\usepackage{amsthm}
\usepackage{dsfont}
\usepackage{lineno}
\usepackage{subfigure}
\usepackage{graphicx}
\usepackage{multirow}
\usepackage{color}
\usepackage{xspace}
\usepackage{pdfsync}
\usepackage{hyperref} 
\usepackage{algorithmicx}
\usepackage{algpseudocode}
\usepackage{algorithm}
\usepackage{mathtools}
\usepackage{enumitem}
\graphicspath{{../Figures/}}

\DeclareMathOperator*{\argmin}{argmin}
\DeclareMathOperator*{\argmax}{argmax}

\newcommand{\bq}{\begin{equation}}
\newcommand{\eq}{\end{equation}}
\newcommand{\R}{\mathbb{R}}
\newcommand{\abs}[1]{\left\vert#1\right\vert}

\newcommand{\G}{\mathcal{G}}

\newcommand{\Dt}{\mathcal{D}}

\newcommand{\Sf}{\mathcal{S}}
\newcommand{\MA}{Monge-Amp\`ere\xspace}

\algnewcommand{\LineComment}[1]{\State \(\triangleright\) #1}

\newtheorem{theorem}{Theorem}
\newtheorem{thm}{Theorem}

\theoremstyle{lemma}

\newtheorem{lem}[theorem]{Lemma}

\newtheorem{cor}[theorem]{Corollary}
\newtheorem{definition}[theorem]{Definition}

\newtheorem{rem}[theorem]{Remark}
\newtheorem{hypothesis}[theorem]{Hypothesis}

\theoremstyle{remark}
\newtheorem{examp}{Example}

\makeatletter
\newcommand\appendix@section[1]{%
\refstepcounter{section}%
\orig@section*{Appendix \@Alph\c@section: #1}%
}
\let\orig@section\section
\g@addto@macro\appendix{\let\section\appendix@section}
\makeatother

\begin{document}

\title[Approximation of surfaces of prescribed Gaussian curvature]{Convergent approximation of non-continuous surfaces of prescribed Gaussian curvature}

\author{Brittany D. Froese}
\address{Department of Mathematical Sciences, New Jersey Institute of Technology, University Heights, Newark, NJ 07102}
\email{bdfroese@njit.edu}
\thanks{This work was partially supported by NSF DMS-1619807}

\begin{abstract}
We consider the numerical approximation of surfaces of prescribed Gaussian curvature via the solution of a fully nonlinear partial differential equation of Monge-Amp\`ere type.  These surfaces need not be continuous up to the boundary of the domain and the Dirichlet boundary condition must be interpreted in a weak sense.  As a consequence, sub-solutions do not always lie below super-solutions, standard comparison principles fail, and existing convergence theorems break down.  By relying on a geometric interpretation of weak solutions, we prove a relaxed comparison principle that applies only in the interior of the domain.  We provide a general framework for proving existence and stability results for consistent, monotone finite difference approximations and modify the Barles-Souganidis convergence framework to show convergence in the interior of the domain.  We describe a convergent scheme for the prescribed Gaussian curvature equation and present several challenging examples to validate these results.
\end{abstract}

\date{\today}    
\maketitle

\section{Introduction}\label{sec:intro}

The Gaussian curvature of a hypersurface is given by the product of the principle curvatures of the surface. When a hypersurface in $\R^{n+1}$ can be locally characterised as the graph of a $C^2$ function $(x,u(x))$, the Gaussian curvature at the point $x$ is given by
\bq\label{eq:curvature} \kappa(x) = \frac{\det(D^2u(x))}{(1+\abs{\nabla u(x)}^2)^{(n+2)/2}}. \eq

This characterisation is closely related to the Darboux equation, which can be used to describe the isometric embedding of a Riemannian manifold into $\R^3$~\cite{Blocki,Han_isometric}.  Curvatures of this type also arise in problems involving surface evolution~\cite{OsherSethian},  surface fairing~\cite{ElseyEsedoglu}, image processing~\cite{Sapiro}, and optimal transportation~\cite{oliker2007embedding}.

A widely-studied problem is to find a convex surface $u(x)$ on a convex domain $\Omega\in\R^n$ whose Gaussian curvature is equal to a prescribed function $\kappa:\Omega\to[0,\infty)$~\cite{TrudingerWang_MAReview}.  The Dirichlet problem for the equation of prescribed Gaussian curvature is given by the following fully nonlinear elliptic partial differential equation (PDE).
\bq\label{eq:MA}
\begin{cases}
\det(D^2u(x)) = \kappa(x)(1+\abs{\nabla u(x)}^2)^{(n+2)/2}, & x\in\Omega\\
u(x) = g(x), & x\in\partial\Omega\\
u \text{ is convex.}
\end{cases}
\eq

This belongs to the class of \MA type equations,
\bq\label{eq:MAgen} \det(D^2u(x)) = f(x,u(x),\nabla u(x)). \eq
These have been studied extensively and a wealth of results are available relating to well-posedness, regularity, and applications~\cite{BardiMannucci,CafNirSpruck,Lions_Remarks,TrudingerWang_MAReview,Urbas_MA}.  In general, equations of this type do not possess classical solutions, and it is necessary to rely on some notion of weak solution: either generalised~\cite{Bakelman} or viscosity solutions~\cite{IshiiLions}.  In fact, for the prescribed Gaussian curvature equation, it is easy to construct examples where even the Dirichlet boundary data cannot be enforced in a classical sense.  Instead, the desired weak solution is the supremum of all solutions that lie below the given boundary data.

Recently, the numerical solution of particular \MA equations has received a great deal of attention, with several new methods being proposed.  An early method by Oliker used a geometric argument to construct coarse approximations that converge to the generalised solution of a simple \MA equation~\cite{olikerprussner88}.  Many other recent methods for \MA equations either lack any proof of convergence or require additional regularity assumptions~\cite{BrennerNeilanMA2D,DGnum2006,FengNeilan,DelzannoGrid,LoeperMA,Saumier,SulmanWilliamsRussell}. The powerful Barles-Souganidis convergence framework has recently inspired the development of several monotone (elliptic) approximation schemes~\cite{benamou2014monotone,FO_MATheory,mirebeau2015MA}.  However, complete proofs of convergence to weak solutions are typically not available due to the difficulty of handling boundary conditions.  In particular, the Barles-Souganidis approach requires the PDE, with boundary conditions interpreted in a weak sense, to satisfy a very strong form of a comparison principle.  This comparison principle is often difficult to verify, and in many problems it is actually demonstrably false.
For the non-classical Dirichlet problem considered in this article, it is easy to demonstrate that the form of the comparison principle required by currently available convergence proofs does not hold.

\subsection{Contributions of this work}
The primary goal of this article is to describe a robust framework for approximating generalised surfaces of prescribed Gaussian curvature, with Dirichlet boundary conditions interpreted in a weak sense.   
As a secondary goal, we provide very general results on solution existence and stability for elliptic approximation schemes, as well as results on the convergence of grid functions, which addresses a gap in the literature that has conventionally been handled on a case-by-case basis or simply assumed as a hypothesis.  Our hope is that these contributions will serve as a foundation for producing convergent schemes for more general boundary value problems for fully nonlinear equations.

To accomplish these goals, we address the following specific challenges.

\begin{enumerate}
\item The equation is augmented by an additional condition that the solution be convex.  Thus it is necessary to develop numerical methods that also enforce this condition in an appropriate (approximate) sense.  We rigorously establish that the constrained PDE can be replaced by an equivalent unconstrained PDE that automatically selects the convex solution.  This reduces the problem to the more tractable task of building a convergent numerical method for an unconstrained PDE.
\item The theory of viscosity solutions provides a powerful framework for proving the convergence of numerical methods for fully nonlinear elliptic equations.  In order to make use of this theory, we need to show that generalised surfaces of prescribed Gaussian curvature can be characterised as viscosity solutions of a well-posed PDE.  While it is known how to define viscosity solutions of the Gaussian curvature equation, general uniqueness results are not available for non-continuous solutions.  We use a geometric interpretation of the convex solutions of the PDE to prove new results on the existence and uniqueness of viscosity solutions, as well as their equivalence to generalised surfaces of prescribed Gaussian curvature.
\item Existing convergence proofs rely on a strong form of the comparison principle~\cite{BSnum}, which is actually false in our setting.  However, we prove that this equation satisfies a weaker form of the comparison principle that applies only in the interior of the domain.  We then modify the usual Barles-Souganidis proof to demonstrate convergence of suitable approximation schemes in the interior of the domain (as well as convergence in $L^p$ for $0 < p < \infty$).
\item For the modified Barles-Souganidis framework to apply, it is important that schemes have a solution and that solutions are bounded in $L^\infty$.  We describe a general approach to proving existence and stability of elliptic approximation schemes.  These results are valid under the mild condition that it is possible to construct strict classical sub- and super-solutions of the underlying PDE.  In particular, we show that this is possible for the equation of prescribed Gaussian curvature.
\item With a general convergence framework in place, we turn to the construction of schemes that satisfy the necessary conditions.  For interior points, it is possible to make use of existing numerical methods~\cite{FroeseMeshfreeEigs,FO_MATheory}.  However, when the Dirichlet boundary condition is interpreted in a weak sense, the actual boundary values of a solution are not known \emph{a priori}.  Remarkably, we demonstrate that it is sufficient to enforce the boundary condition in a strong sense.  We show that this satisfies the necessary consistency condition and interior convergence is guaranteed, though a boundary layer is possible when solutions are discontinuous.
\end{enumerate}

We conclude this article by implementing a provably convergent meshfree finite difference method and providing results for several challenging examples.

\section{Weak Solutions}\label{sec:weak}


We begin by reviewing basic notions of weak solution: (1) the generalised solution, which we want to construct and (2) the viscosity solution, which is more amenable to numerical analysis.  One of the goals of this work is to show that these weak solutions are equivalent, so that numerical convergence results for viscosity solutions will also apply to generalised surfaces of prescribed Gaussian curvature.  This equivalence is established in Theorems~\ref{thm:exist} and~\ref{thm:uniqueness}.

\subsection{Generalised solutions}\label{sec:generalised}
The existence of a convex surface with Gaussian curvature $\kappa(x)$ is not guaranteed for arbitrary functions $\kappa \geq 0$.  In particular, the total curvature must be bounded by the volume of the unit ball in $\R^{n}$.
\begin{lem}[Necessary condition for existence~\cite{Bakelman}]\label{lem:condition}
A necessary condition for the existence of a solution of~\eqref{eq:MA} is for the prescribed curvature to satisfy
\bq\label{eq:condition}
\int_\Omega\kappa(x)\,dx  \leq \int_{\R^n} (1+\abs{p}^2)^{-(n+2)/2}.
\eq
\end{lem}

This condition is not sufficient to guarantee that a classical $C^2$ solution exists~\cite{TrudUrbas_Gauss}.  Instead, some notion of weak solution is needed to make sense of solutions of the \MA equation.  One approach is the generalised solution, which defines weak solutions in terms of the measure generated by the subgradient of the solution.  

\begin{definition}[Generalised solution]\label{def:generalised}
A convex function $u:\Omega\to\R$ is a generalised solution of the prescribed Gaussian curvature equation if for every measurable set $E\subset\Omega$
\[ \int_{\partial u(E)}(1+\abs{p}^2)^{-(n+2)/2}\,dp = \int_E \kappa(x)\,dx \]
where $\partial u$ is the subgradient of $u$.
\end{definition}



It may also be impossible to enforce the Dirichlet boundary data in a classical sense.  Instead, Bakelman~\cite{Bakelman} described a weaker notion of Dirichlet boundary conditions for this problem.
\begin{definition}[Weak formulation of boundary conditions]\label{def:weakDirichlet}
A convex function $u $ satisfies the Dirichlet problem~\eqref{eq:MA} if $u$ satisfies the \MA PDE in a generalised sense, 
\bq\label{eq:weakBC}
\limsup\limits_{x\to y}u(x) \leq g(y), \quad y \in\partial\Omega,
\eq
and if $v$ is any other generalised solution of the \MA PDE that also satisfies~\eqref{eq:weakBC} then $v\leq u$ on $\Omega$.
\end{definition}

This weaker notion of Dirichlet boundary conditions leads to an existence result for the problem of prescribed Gaussian curvature.
For clarity, we first state our hypotheses on the data, which will be used throughout this paper.
\begin{hypothesis}[Conditions on data]\label{hyp}
\end{hypothesis}

\begin{enumerate}
\item[(H1)] $\Omega$ is a uniformly convex, bounded, open domain.
\item[(H2)] The boundary data $g\in C^0(\partial\Omega)$.
\item[(H3)] The curvature $\kappa:\bar{\Omega}\to [0,\infty)$ is continuous and bounded.
\item[(H4)] The data satisfies the strict compatibility condition
\[\int_\Omega\kappa(x)\,dx  < \int_{\R^n} (1+\abs{p}^2)^{-(n+2)/2}\,dp.\]
\end{enumerate}

\begin{lem}[Existence of generalised solution~{\cite[Theorem~1]{Bakelman}}]\label{lem:exist}
If Hypothesis~\ref{hyp} holds, the \MA equation~\eqref{eq:MA} has a unique generalised solution that satisfies the Dirichlet boundary conditions in the weak sense.
\end{lem}

\begin{rem}
The generalised solution is convex and therefore continuous in the interior $\Omega$, but need not be continuous up to the boundary even if the Dirichlet data~$g$ is continuous.  For example, when $\kappa(x)=0$, the solution is the convex envelope of the boundary data~\cite{ObermanCE}, which need not be continuous up to the boundary~\cite{KruskalCE}.
\end{rem}

\subsection{Viscosity solutions}\label{sec:visc}
We will make use of an alternative (equivalent) form of weak solution, the viscosity solution, which will inform the convergent approximation schemes we construct.

The \MA equation belongs to a class of PDEs known as degenerate elliptic equations, which take the form
\[ F(x,u(x),\nabla u(x),D^2u(x)) = 0. \]

\begin{definition}[Degenerate elliptic]\label{def:elliptic}
The operator
$F:\Omega\times\R\times\R^n\times\Sf^n\to\R$
is \emph{degenerate elliptic} if 
\[ F(x,u,p,X) \leq F(x,v,p,Y) \]
whenever $u \leq v$ and $X \geq Y$.
\end{definition}

The notion of the viscosity solution has become a very powerful tool for analysing fully nonlinear degenerate elliptic PDEs~\cite{CIL}.  The definition relies on a maximum principle argument that moves derivatives onto smooth test functions.  The usual definition of a viscosity solution must be modified slightly for the \MA equation, which is elliptic only the space of convex functions.  This requires a slight alteration to the test functions that must be checked.  


For brevity, we introduce the notation
\bq\label{eq:grad}
R(p) = (1+\abs{p}^2)^{(n+2)/2}.
\eq
Then we can denote the operator $F:{\Omega}\times\R\times\R^n\times\Sf^n\to\R$ corresponding to equation~\eqref{eq:MA} by
\bq\label{eq:MAop} F(x,z,p,X) = -\det(X) + \kappa(x)R(p).\eq

Convex viscosity solutions of the equation 
\bq\label{eq:PDE} F(x,u(x),\nabla u(x),D^2 u(x))  = 0 \eq
are defined as follows.
\begin{definition}[Viscosity sub-solution]\label{def:subsolution}
A convex upper semi-continuous function $u$ is a \emph{viscosity sub-solution} of~\eqref{eq:PDE} in $\Omega$ if for every $\phi\in C^2({\Omega})$, whenever $u-\phi$ has a local maximum at $x \in {\Omega}$, then
\[ 
F(x,u(x),\nabla \phi(x),D^2\phi(x)) \leq   0 .
\]
\end{definition}
\begin{definition}[Viscosity super-solution]\label{def:supersolution}
A convex lower semi-continuous function $u$ is a \emph{viscosity super-solution} of~\eqref{eq:PDE} in $\Omega$ if for every $\phi\in C^2({\Omega})$, whenever $u-\phi$ has a local minimum  at $x \in {\Omega}$ and $D^2\phi(x) \geq 0$, then
\[ 
F(x,u(x),\nabla \phi(x),D^2\phi(x)) \geq   0 .
\]
\end{definition}

\begin{rem}
In the definition of the viscosity super-solution, the space of test functions is restricted to smooth, convex functions.  In the definition of the sub-solution, it is not necessary to require test functions to be convex, although local convexity near $x$ follows automatically from the fact that $u-\phi$ has a maximum.  We choose to use different test function spaces for the sub- and super-solutions here because this will allow us to maintain the same test function spaces when we consider a modified version of the operator~\eqref{eq:MAconvex}.
\end{rem}


\begin{definition}[Viscosity solution]\label{def:viscosity1}
A convex function $u:\Omega\to\R$ is a \emph{viscosity solution} of~\eqref{eq:PDE} if it is both a sub-solution and a super-solution.
\end{definition}

For the moment we ignore the issue of boundary conditions and focus on the behaviour of solutions in the interior of the domain.
On open sets, viscosity solutions are equivalent to generalised solutions.  

\begin{thm}[Equivalence of weak solutions]\label{thm:equivalence1}
Let $\kappa:\Omega\to[0,\infty)$ be continuous and bounded.  Then a convex function $u$ is a viscosity solution of the prescribed Gaussian curvature equation~\eqref{eq:curvature} if and only if it is a generalised solution.
\end{thm}

While this equivalence is known for Monge-Amp\`ere type equations~\cite[\S~4.1.4]{Villani}, a detailed proof in the particular case of the prescribed Gaussian curvature equation is not readily found in the literature.  For completeness, we provide a proof in Appendix~\ref{app:equivalence}.  We note that this proof makes use of several classical concepts that will be introduced throughout the paper.  However, it does not depend on any of our key theorems except for Theorem~\ref{lem:subgradOrder}, which is itself a simple consequence of the definition of the subgradient.

\section{Convexity constraint}\label{sec:convexity}

The equation~\eqref{eq:MA} for prescribed Gaussian curvature is elliptic only on the space of convex functions, and convexity of the solution needs to be included as an additional constraint.  
It is not immediately evident how to develop practical numerical methods that enforce this constraint.  Instead, as in~\cite{FroeseTransport}, we wish to absorb this constraint into the PDE operator to produce an equivalent equation that is globally elliptic and automatically selects the convex solution.  
In that work, the determinant of the Hessian was re-expressed as
\[ {\det}^+(D^2u) = \prod\limits_{j=1}^n \max\{\lambda_j(D^2u),0\} + \sum\limits_{j=1}^n \min\{\lambda_j(D^2u),0\} \]
where
\[ \lambda_1(D^2u) \leq \ldots \leq \lambda_n(D^2u) \]
are the eigenvalues of the matrix $D^2u$.
For smooth $u$, this is equivalent to $\det(D^2u)$ when $u$ is an admissible (convex) function, and produces a negative result otherwise.  Equivalence in the sense of viscosity solutions was not previously established.

Here, we will use an alternate reformulation and replace the PDE operator~\eqref{eq:PDE} with the following.
\bq\label{eq:MAconvex}
F(x,z,p,X) = 
\max\left\{-\prod\limits_{j=1}^n\max\{\lambda_j(X),0\} + \kappa(x)R(p), -\lambda_1(X)\right\}.
\eq
We remark that since the desired solution formally has a positive semi-definite Hessian (i.e. $0 \leq \lambda_1 \leq \ldots \leq \lambda_n$), the equation
\bq\label{eq:constraint1} -\lambda_1(X) = 0 \eq
defines the boundary of the constraint set.  Equation~\eqref{eq:constraint1} is equivalent to the constrained equation
\bq\label{eq:constrain2} -\det(X) = 0, \quad 0 \leq \lambda_1(X) \leq \ldots \leq \lambda_n(X), \eq
which is the prescribed Gaussian curvature equation corresponding to $\kappa = 0$.  Intuitively, then, setting the operator~\eqref{eq:MAconvex} equal to zero requires that either
\[ \prod\limits_{j=1}^n\max\{\lambda_j(X),0\} = \kappa(x)R(p), \quad \lambda_1(X) \geq 0 \]
or
\[ \lambda_1(X) = 0, \quad 0 = \prod\limits_{j=1}^n\max\{\lambda_j(X),0\} \geq \kappa(x)R(p) \geq 0. \]
In either case, setting $X = D^2u$, we recover a convex solution of the prescribed Gaussian curvature equation.

We now rigorously establish this equivalence in the context of viscosity solutions.
For brevity in the following exposition, we will define the function ${\det}^+:\Sf^n\to[0,\infty)$ by
\bq\label{eq:detP}
{\det}^+(X) \equiv \prod\limits_{j=1}^n\max\{\lambda_j(X),0\} = \begin{cases}\det(X), & X \geq 0\\0, & \text{otherwise}.  \end{cases}
\eq

The definition of the viscosity solution for this operator is the same as Definition~\ref{def:viscosity1} except that we allow for the possibility of non-convex sub(super)-solutions.  

\begin{definition}[Viscosity sub-solution]\label{def:subsolution}
An upper semi-continuous function $u$ is a \emph{viscosity sub-solution} of~\eqref{eq:MAconvex} in $\Omega$ if for every $\phi\in C^2({\Omega})$, whenever $u-\phi$ has a local maximum at $x \in {\Omega}$, then
\[ 
F(x,u(x),\nabla \phi(x),D^2\phi(x)) \leq   0 .
\]
\end{definition}
\begin{definition}[Viscosity super-solution]\label{def:supersolution}
A lower semi-continuous function $u$ is a \emph{viscosity super-solution} of~\eqref{eq:MAconvex} in $\Omega$ if for every $\phi\in C^2({\Omega})$, whenever $u-\phi$ has a local minimum  at $x \in {\Omega}$ and $D^2\phi(x) \geq 0$, then
\[ 
F(x,u(x),\nabla \phi(x),D^2\phi(x)) \geq   0 .
\]
\end{definition}

\begin{definition}[Viscosity solution]\label{def:viscosity1}
A function $u:\Omega\to\R$ is a \emph{viscosity solution} of~\eqref{eq:MAconvex} if it is both a sub-solution and a super-solution.
\end{definition}

\begin{thm}[Global ellipticity]\label{thm:elliptic}
The convexified \MA operator~\eqref{eq:MAconvex} is globally degenerate elliptic.
\end{thm}
\begin{proof}
It is sufficient to show that the operator is a non-increasing function of the eigenvalues $\lambda_1(X), \ldots, \lambda_n(X)$~\cite{caffarelli_eigs}.

Both the modified function ${\det}^+(X)$ and the smallest eigenvalue $\lambda_1(X)$ are non-decreasing functions of the eigenvalues of the symmetric matrix $X$, and thus the combination
\[ \max\left\{-\prod\limits_{j=1}^n\max\{\lambda_j(X),0\} + \kappa(x)R(p), -\lambda_1(X)\right\} \]
is a non-increasing function of the eigenvalues and the operator is degenerate elliptic.
\end{proof}

A key advantage of this formulation is that it automatically forces the solution to be convex instead of requiring this condition to be included as an additional constraint in the definition.  In particular, all sub-solutions of this new formulation are convex, which means that viscosity solutions must also be convex.

\begin{lem}[Sub-solutions are convex]\label{lem:subConvex}
Let $u$ be an upper semi-continuous sub-solution of the convexified \MA equation~\eqref{eq:MAconvex}.  Then $u$ is convex.
\end{lem}
\begin{proof}
Choose $x_0\in\Omega$ and $\phi\in C^2$ such that $u-\phi$ has a local maximum at $x_0$.  Since $u$ is a sub-solution of~\eqref{eq:MAconvex},
\[ \max\{-{\det}^+(D^2\phi(x_0)) + \kappa(x_0)R(\nabla\phi(x_0)),-\lambda_1(D^2\phi(x_0))\} \leq 0. \]
An immediate consequence of this is that $-\lambda_1(D^2\phi(x_0)) \leq 0$
 and therefore $u$ is also a sub-solution of
\[ -\lambda_1(D^2u(x)) = 0. \]
This is precisely the hypothesis of~\cite[Theorem~1]{ObermanCE}, which establishes the convexity of $u$. 
\end{proof}

One of our goals is to establish that viscosity solutions of the convexified equation are equivalent to viscosity solutions of the original equation.  The two different notions of sub- and super-solutions are not strictly equivalent since the convexified operator allows for non-convex super-solutions.  However, we can demonstrate that the concepts of sub- and super-solutions are equivalent on the set of convex functions.  Combined with the fact that sub-solutions are always convex, this is sufficient for proving that the two different notions of viscosity solution are equivalent (Theorem~\ref{thm:equivalence}).


\begin{lem}\label{lem:sub12}
A convex function $u$ is a sub-solution of the original \MA equation~\eqref{eq:PDE} if and only if it is a sub-solution of the convexified \MA equation~\eqref{eq:MAconvex}.
\end{lem}
\begin{proof}
Choose any $\phi\in C^2$ and $x_0\in\Omega$ such that $u-\phi$ has a local maximum at $x_0$.  Since $u$ is convex, there exists some $q\in\partial u(x_0)$ and we can define the supporting hyperplane 
\[ p(x) = u(x_0) + q\cdot(x-x_0) \leq u(x). \]
For $x$ near $x_0$ we have
\[ p(x)-\phi(x) \leq u(x)-\phi(x) \leq u(x_0)-\phi(x_0) = p(x_0)-\phi(x_0). \]
Thus $p-\phi$ is also maximised at $x_0$, which requires
 $D^2\phi(x_0)  \geq D^2p(x_0) = 0$.  This in turn implies that $-\lambda_1(D^2\phi(x_0)) \leq 0$ and ${\det}^+(D^2\phi(x_0)) = {\det}(D^2\phi(x_0))$.  Under these constraints, the condition that $u$ is a sub-solution of~\eqref{eq:MAconvex}:
\[ \max\{-{\det}^+(D^2\phi(x_0)) + \kappa(x_0)R(\nabla\phi(x_0)), -\lambda_1(D^2\phi(x_0))\} \leq 0, \]
is equivalent to the condition that $u$ is a sub-solution of~\eqref{eq:PDE}:
\[
-\det(D^2\phi(x_0)) + \kappa(x_0)R(\nabla\phi(x_0)) \leq 0. 
\]
\end{proof}


\begin{lem}\label{lem:super12}
A convex function $u$ is a super-solution of the original \MA equation~\eqref{eq:PDE} if and only if it is a super-solution of the convexified \MA equation~\eqref{eq:MAconvex}.
\end{lem}
\begin{proof}
Choose any $\phi\in C^2$ and $x_0\in\Omega$ such that $u-\phi$ is minimised at $x_0$ and $D^2\phi(x_0) \geq 0$.  
As in the previous lemma, this restriction on $\phi$ ensures that the condition
\[ \max\{-{\det}^+(D^2\phi(x_0)) + \kappa(x_0)R(\nabla\phi(x_0)), -\lambda_1(D^2\phi(x_0))\} \geq 0 \]
is equivalent to 
\[
-\det(D^2\phi(x_0)) + \kappa(x_0)R(\nabla\phi(x_0)) \geq 0. 
\]
\end{proof}

Lemmas~\ref{lem:subConvex}-\ref{lem:super12} lead immediately to the equivalence of the two different formulations of the \MA equation.  

\begin{thm}[Equivalence of viscosity solutions]\label{thm:equivalence}
A function $u$ is a convex viscosity solution of the original \MA equation~\eqref{eq:PDE} if and only if it is a viscosity solution of the convexified \MA equation~\eqref{eq:MAconvex}.
\end{thm}

\section{Generalised Dirichlet Problem}\label{sec:dirichlet}
One of the advantages of working with viscosity solutions is that boundary conditions can be included in the operator, which allows for a weak interpretation of Dirichlet boundary conditions~\cite{CIL}.  A key result that we will build to throughout this section is that a viscosity interpretation of the boundary conditions recovers the desired weak interpretation given in Definition~\ref{def:weakDirichlet}.

\subsection{Boundary conditions}\label{sec:bc}


The weak solution we are seeking is permitted to lie below the Dirichlet data.  Thus it is necessary to relax the usual notion of super-solution so as to allow these to satisfy this weak interpretation of the boundary conditions.  
In order to accomplish this, we introduce the modified PDE operators
\bq\label{eq:MAlower}
F_*(x,z,p,X) = \begin{cases}
\max\{-{\det}^+(X) + \kappa(x)R(p),-\lambda_1(X)\}, & x\in\Omega\\
\min\{z-g(x),\max\{-{\det}^+(X) + \kappa(x)R(p),-\lambda_1(X)\}\}, & x\in\partial\Omega, 
\end{cases}
\eq
\bq\label{eq:MAupper}
F^*(x,z,p,X) = \begin{cases}
\max\{-{\det}^+(X) + \kappa(x)R(p),-\lambda_1(X)\}, & x\in\Omega\\
\max\{z-g(x),\max\{-{\det}^+(X) + \kappa(x)R(p),-\lambda_1(X)\}\}, & x\in\partial\Omega.
\end{cases}
\eq

Sub- and super-solutions of the Dirichlet problem~\eqref{eq:MA} are defined as follows.

\begin{definition}[Viscosity sub(super)-solutions]\label{def:subsuperWeak}
A bounded upper (lower) semi-continuous function $u$ is a \emph{viscosity sub(super)-solution} of~\eqref{eq:MA} if for every $\phi\in C^2(\bar{\Omega})$, whenever $u-\phi$ has a local maximum (minimum)  at $x \in \bar{\Omega}$, then
\[ 
F_*^{(*)}(x,u(x),\nabla \phi(x),D^2\phi(x)) \leq(\geq)  0 .
\]
\end{definition}

\begin{rem}
In the definition of a super-solution, it is sufficient to use test functions $\phi$ satisfying $D^2\phi(x_0) > 0$.  For other smooth test functions, $-\lambda_1(D^2\phi(x_0)) > 0$ and the conditions $F^* \geq 0$ is automatically satisfied regardless of the behaviour of $u$.
\end{rem}

Originally, we required a viscosity solution to be both upper and lower semi-continuous, and therefore continuous.  However, viscosity solutions of the Dirichlet problem need not be continuous up to the boundary, so this condition needs to be relaxed.  We can do this by making use of the semi-continuous envelopes of a candidate solution~\cite{CIL}.

\begin{definition}[Semi-continuous envelopes]\label{def:envelopes}
 Let $u:{\Omega}\to\R$ be a bounded function.  Then for $x\in\bar{\Omega}$ its upper and lower semi-continuous envelopes are defined respectively by
\[ u^*(x) = \limsup\limits_{y\to x}u(y), \quad u_*(x) = \liminf\limits_{y\to x} u(y). \]
\end{definition}

\begin{definition}[Viscosity solution]\label{def:viscosity}
A bounded function $u:\Omega\to\R$ is a \emph{viscosity solution} of~\eqref{eq:MA} if $u^*$ is a sub-solution and $u_*$ is a super-solution.
\end{definition}

\begin{examp}
To illustrate the non-classical nature of the Dirichlet condition, we consider the one-dimensional Gaussian curvature equation with constant unit curvature: 
\bq\label{eq:1d}
F(x,u,u_x,u_{xx}) 
=
\begin{cases}
-u_{xx} + (1+u_x^2)^{3/2}, & x\in (0,1)\\
u+1, & x  = 0\\
u - 1, & x = 1.
\end{cases}
\eq
Note that in this case,
\[ \int_0^1 \kappa(x)\,dx = 1 < 2 = \int_{-\infty}^\infty (1+p^2)^{-3/2}\,dp, \]
so this problem satisfies the existence and uniqueness requirements of Lemma~\ref{lem:exist}.
We claim that the viscosity solution lies on the surface of the unit ball,
\[ u(x) = -\sqrt{1-x^2}, \]
which does not satisfy the Dirichlet boundary condition $u(1) = 1$.  See Figure~\ref{fig:super1D}.

A simple calculation verifies that the equation is satisfied in a classical sense on $[0,1)$.  It remains to verify the conditions for a viscosity solution at $x=1$.

Clearly, $u$ is a sub-solution since at $x=1$, any test function $\phi$ will satisfy
\[ F_*(1,u(1),\phi_x(1),\phi_{xx}(1)) \leq u(1) - 1 = -1 < 0. \]

Next we check the super-solution property at $x=1$.  To do so, we need to consider all functions $\phi\in C^2$ such that $u-\phi$ has a minimum at $x = 1$.  However, $u_{x}\to\infty$ as $x\to1^-$, which means that no such smooth test function exists and there is nothing to check.
\hfill $\square$
\end{examp}

\begin{figure}[tbhp]
  \centering
  \subfigure[]{\label{fig:super1D}\includegraphics[width=0.42\textwidth]{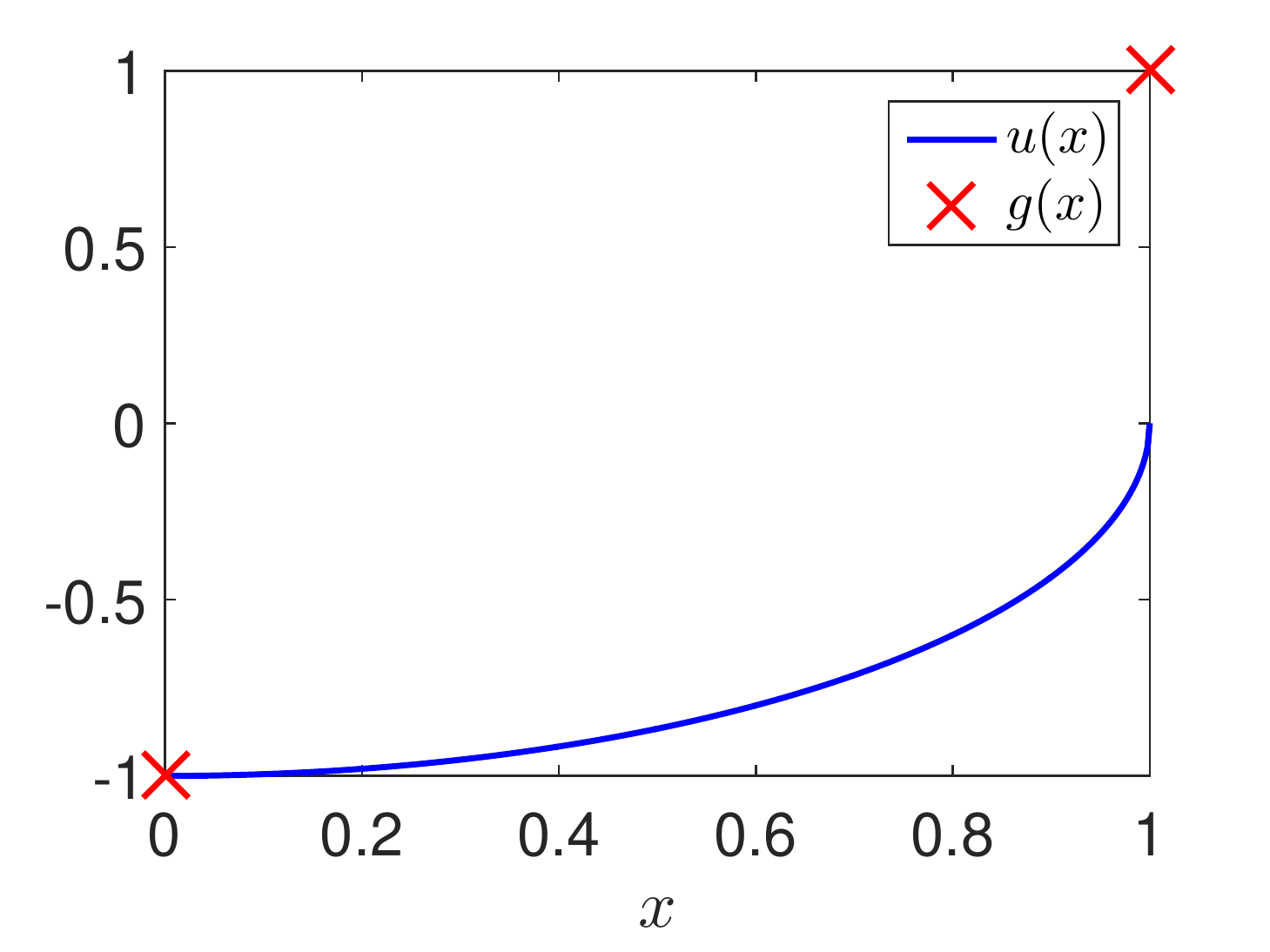}}
  \subfigure[]{\label{fig:sub1D}\includegraphics[width=0.42\textwidth]{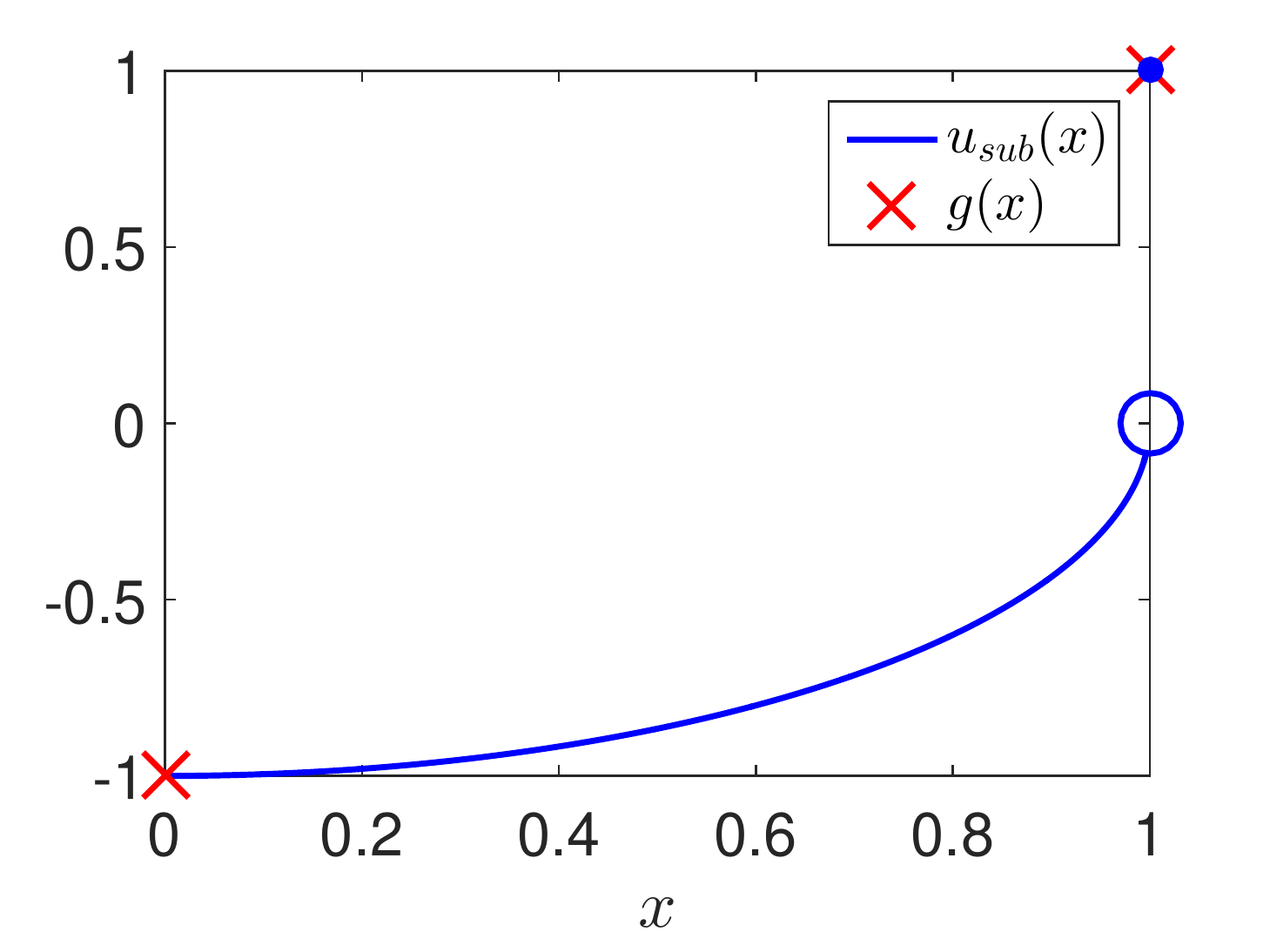}}
  \caption{\subref{fig:super1D}~A viscosity solution with constant Gaussian curvature that does not achieve the Dirichlet boundary conditions and \subref{fig:sub1D}~a sub-solution that lies above this viscosity solution.}
  \label{fig:1d}
\end{figure}

While the viscosity formulation does not require solutions to achieve the Dirichlet data, it does place some conditions on the behaviour of the solution near the boundary.  These conditions are outlined in the following two lemmas and Corollary~\ref{cor:viscBC}.  In particular, we find that sub-solutions must lie below the Dirichlet data, except possibly at a small number of points.  Viscosity solutions are always bounded above by the boundary data, though they are permitted to lie strictly below the Dirichlet boundary conditions.  This means that super-solutions cannot be required to lie above the Dirichlet boundary conditions; however, this traditional property can be violated only at points where the gradient is unbounded.

\begin{lem}[Behaviour of sub-solutions at boundary]\label{lem:subBC}
Let $u$ be an upper semi-continuous sub-solution of~\eqref{eq:MA} with data satisfying Hypothesis~\ref{hyp}.  Then $(u_*)^*\leq g$ on $\partial\Omega$.
\end{lem}

\begin{proof}
Choose any $x_0\in\partial\Omega$ and small $\epsilon>0$.  Since $\Omega$ is convex, there exists a supporting hyperplane to the domain at $x_0$.  We  let $n(x_0)$ be the unit outward normal to any such hyperplane.  Since $\Omega$ is uniformly convex, there exists some $\alpha>0$ such that for any $x\in\bar{\Omega}$ with $\abs{x-x_0}$ sufficiently small,
\[ -n(x_0)\cdot(x-x_0) \geq \alpha\abs{x-x_0}^2. \]

Denote by $B$ the open ball $B(x_0,\epsilon)$.  For any $x\in\partial B\cap\bar{\Omega}$ and sufficiently large $\gamma>0$,
\[ P(x) \equiv u(x_0)-\gamma n(x_0)\cdot(x-x_0) \geq u(x_0) + \gamma\alpha\epsilon^2>\max\limits_{\partial B\cap\bar{\Omega}}u. \]

Since $u$ is upper semi-continuous, there exists some
\[ z \in \argmax\limits_{\bar{B}\cap\bar{\Omega}}\{u-P\}. \]
We note that $u-P<0$ on $\partial B$ and $u(x_0)-P(x_0) = 0$.  Thus $z\notin\partial B$ and $u-P$ has a local maximum at $z$.

Consider any $x\in B\cap\Omega$.  As the intersection of two convex sets, $B\cap\Omega$ is also convex.  Since $x$ is in the interior of this convex set, it can be expressed as $\lambda_1 x_1 + \lambda_2 x_2$ for some $x_1\in\partial B\cap\Omega$, $ x_2 \in B\cap\Omega$ and $\lambda_1, \lambda_2 > 0$ with $\lambda_1+\lambda_2 = 1$.  Since $u$ is convex (Lemma~\ref{lem:subConvex}) and $P$ is affine, we can calculate
\begin{align*}
u(x)-P(x) &= u(\lambda_1 x_1 + \lambda_2 x_2) - P(\lambda_1 x_1 + \lambda_2 x_2)\\
 &\leq \lambda_1(u(x_1)-P(x_1)) + \lambda_2(u(x_2)-P(x_2))\\
 &< u(z)-P(z).
\end{align*}
Therefore $z\in B\cap\partial\Omega$.

As $\Omega$ is uniformly convex, there exists $\beta>0$ such that whenever $x\in\bar{\Omega}$,
\[ (x-z)\cdot n(z) \leq -\beta\abs{x-z}^2. \]

Define the test function
\[ \phi(x) = P(x)-(x-z)\cdot n(z) - \beta\abs{x-z}^2 \in C^2. \]
We notice that
\[ u(x)-\phi(x) \leq u(x)-P(x) \leq u(z)-P(z) = u(z)-\phi(z). \]
Thus $u-\phi$ has a local maximum at $z$.  Since $u$ is a sub-solution, this requires
\[ F_*(z,u(z),\nabla\phi(z),D^2\phi(z)) \leq 0. \]

However, by construction, $\lambda_1(D^2\phi(z)) = -2\beta < 0$ so that
\[ \max\{-{\det}^+(D^2\phi(z)) + \kappa(z)R(\nabla\phi(z)),-\lambda_1(D^2\phi(z))\} > 0. \]
Since $u$ is a sub-solution, we require
\[ u(z)-g(z) \leq 0.\]

We have shown that for any $\epsilon>0$, there exists some $z\in B(x_0,\epsilon)\cap\partial\Omega$ such that $u(z) \leq g(z)$.  Since $g$ is continuous, we conclude that $u_*(x_0) \leq g(x_0)$ for $x_0\in\partial\Omega$.

Since $u$ is convex, $u_* = u = u^*$ in $\Omega$ and $u_*$ is convex on $\bar{\Omega}$.  
Consider some $x_0\in\partial\Omega$.  For any $\epsilon>0$, there exists some $x_\epsilon\in\Omega$ such that $x_\epsilon\to x_0$ and 
\[ u(x_\epsilon) \leq u_*(x_0) + \epsilon \leq g(x_0)+\epsilon. \]
In addition, there exists some $y_\epsilon\in B(x_0,\abs{x_0-x_\epsilon})\cap\Omega$ such that
\[ u(y_\epsilon) \geq (u_*)^*(x_0)-\epsilon. \]
Finally, we can define $z_\epsilon\in B(x_0,\abs{x_0-x_\epsilon})\cap\partial\Omega$ such that for some $\lambda_1,\lambda_2\geq 0$ with $\lambda_1+\lambda_2=1$,  $y_\epsilon = \lambda_1x_\epsilon + \lambda_2z_\epsilon$.  Then we can compute
\[ (u_*)^*(x_0)-\epsilon \leq u(y_\epsilon) \leq \lambda_1u(x_\epsilon)+\lambda_2u_*(z_\epsilon) \leq \lambda_1\left(g(x_0)+\epsilon\right) + \lambda_2g(z_\epsilon). \]

Taking $\epsilon\to0$ we obtain
\[ (u_*)^*(x_0) \leq g(x_0).  \]
\end{proof}

\begin{lem}[Behaviour of super-solutions at boundary]\label{lem:superBC}
Let $u$ be a lower semi-continuous super-solution of~\eqref{eq:MA} with data satisfying Hypothesis~\ref{hyp}.  Then at each $x_0\in\partial\Omega$, either $u(x_0) \geq g(x_0)$ or the subgradient $\partial u(x_0)$ is empty.
\end{lem}

\begin{proof}
Let $x_0\in\partial\Omega$ and suppose that both $u(x_0) < g(x_0)$ and there exists some $p\in\partial u(x_0)$.  Consider any supporting hyperplane to the domain at $x_0$ and let $n$ be the unit outward normal to this hyperplane.  
Since $\Omega$ is uniformly convex, there exists some constant $\alpha>0$ such that for small enough $\abs{x-x_0}$ with $x\in\bar{\Omega}$,
\[ (x-x_0)\cdot n \leq -\alpha\abs{x-x_0}^2. \]

Now we choose any $\gamma > 0$ and consider the test function
\[ \phi(x) = u(x_0) + p\cdot(x-x_0) +  (x-x_0)\cdot n + \frac{\alpha}{2}\abs{x-x_0}^2 + \frac{\gamma}{2}\left((x-x_0)\cdot n\right)^2.  \]
By the definition of $p$, we have
\[ u(x_0) + p\cdot(x-x_0) \leq u(x). \]
From the definition of $\alpha$ we know that
\[ \frac{1}{2} (x-x_0)\cdot n + \frac{\alpha}{2}\abs{x-x_0}^2 \leq 0. \]
Finally, as long as $\abs{x-x_0} < 1/\gamma$ we have
\[  \frac{1}{2} (x-x_0)\cdot n + \frac{\gamma}{2}\left((x-x_0)\cdot n(x_0)\right)^2 \leq 0.\]
Putting these results together, we obtain
\[ \phi(x) \leq u(x) \]
near $x_0$, with $\phi(x_0) = u(x_0)$.  Thus $u-\phi$ has a local minimum at $x_0$.

We also note that $\phi\in C^2$ and
\begin{align*}
\nabla\phi(x_0) &= p + n,\\
D^2\phi(x_0) &= \alpha I + \gamma nn^T > 0.
\end{align*}

Then for sufficiently large $\gamma$:
\begin{align*}
-{\det}^+(D^2\phi(x_0)) + \kappa(x_0)R(\nabla\phi(x_0)) = -\det(\alpha I + \gamma nn^T) + \kappa(x_0)R(p+ n)&<0,\\
-\lambda_1(D^2\phi(x_0))&<0,\\
u(x_0)-g(x_0) &<0
\end{align*}
so that
\[ F^*(x_0,u(x_0),\nabla\phi(x_0),D^2\phi(x_0))<0, \]
which contradicts the fact that $u$ is a super-solution.
\end{proof}

As an immediate consequence of Lemmas~\ref{lem:subBC}-\ref{lem:superBC} and Definition~\ref{def:viscosity}, we obtain the sense in which a viscosity solution satisfies the Dirichlet boundary conditions.

\begin{cor}[Boundary conditions for viscosity solutions]\label{cor:viscBC}
Let $u$ be a viscosity solution of~\eqref{eq:MA} with data satisfying Hypothesis~\ref{hyp}.  Then at every $x_0\in\partial\Omega$ either $u_*(x_0) = u^*(x_0) = g(x_0)$ or $u_*(x_0)\leq u^*(x_0)\leq g(x_0)$ with $\partial u_*(x_0)$ empty.
\end{cor}

\subsection{Existence of a viscosity solution}\label{sec:existence}

Next we begin to establish the well-posedness of this viscosity formulation.    This section culminates in Theorem~\ref{thm:exist}, which demonstrates that at least one viscosity solution exists---in particular, the generalised solution is a viscosity solution.

We start by describing the Perron method for constructing a viscosity solution, which satisfies the Dirichlet boundary conditions in the weak sense and need not be continuous.  

\begin{thm}[Perron construction of viscosity solution]\label{thm:perron}
Assume that $\Omega$, $g$, and $\kappa$ satisfy Hypothesis~\ref{hyp}.  If $u_1$ is an upper semi-continuous sub-solution and $u_2$ a lower semi-continuous super-solution with $u_1 \leq u_2$ on $\bar{\Omega}$ then
\[ w(x) = \sup\{W:\Omega\to\R \mid u_1 \leq W \leq u_2, \, W^* \text{ is a sub-solution}\} \]
is a viscosity solution of~\eqref{eq:MA}.
\end{thm}

The result is technical and essentially equivalent to~\cite[Theorem~4.1]{BardiMannucci}.  To keep the key contributions of this section clear, we postpone the proof to Appendix~\ref{app:perron}.

A simple consequence of this technique is the fact that the maximal sub-solution is a viscosity solution.

\begin{cor}[Maximal sub-solution]\label{cor:maxSub}
Let $\Omega$, $g$, and $\kappa$ satisfy Hypothesis~\ref{hyp}.  Suppose that the set
\[ U = \{u:\Omega\to\R\mid u^* \text{ is a sub-solution of~\eqref{eq:MA}}\} \]
is non-empty.  Then
\[ w = \sup\{u\in U\}  \]
is a viscosity solution of~\eqref{eq:MA}.
\end{cor}

\begin{proof}
Choose any $u\in U$ and let $v = c$ be a constant function with $c \geq \sup g$.  Clearly $u \leq v$ since all sub-solutions are convex and lie below the Dirichlet data (Lemma~\ref{lem:subConvex}).  We claim that $v$ is a super-solution since it lies above the Dirichlet data and, in the interior, there are no admissible test functions satisfying $D^2\phi(x_0)>0$.  Then by Perron's method (Theorem~\ref{thm:perron}),
\begin{align*} w(x) &\equiv \sup\{W\mid u \leq W \leq v, W^* \text{ is a sub-solution} \}\\ &= \sup\{W \mid W^* \text{ is a sub-solution}\} \end{align*}
is a viscosity solution.
\end{proof}

\begin{thm}[Existence of viscosity solution]\label{thm:exist}
Under Hypothesis~\ref{hyp}, the generalised solution of~\eqref{eq:MA} is also a viscosity solution of~\eqref{eq:MA}.
\end{thm}

\begin{proof}
Let $u$ be the generalised solution of the Dirichlet problem.  By Theorems~\ref{thm:equivalence1} and~\ref{thm:equivalence}, $u$ is  a viscosity solution in~$\Omega$.  Since additionally $\limsup u \leq g$ on $\partial\Omega$, $u^*$ is a viscosity sub-solution of the Dirichlet problem.

The existence of a viscosity sub-solution ensures the existence of a maximal sub-solution~$w$ by Corollary~\ref{cor:maxSub}, with $u \leq w$.  

Again by Theorems~\ref{thm:equivalence1} and~\ref{thm:equivalence}, $w$ is a generalised solution in $\Omega$.  Since $w \leq g$ on $\partial\Omega$ and $u$ is the maximal such generalised solution, we must have $w \leq u$.

We conclude that the generalised solution to the Dirichlet problem is also a viscosity solution of the Dirichlet problem.
\end{proof}

\subsection{Uniqueness and comparison}\label{sec:comparison}

Next, we need to demonstrate that this viscosity solution is unique.  
We first note that the condition
\[ \int_\Omega\kappa(x)\,dx  < \int_{\R^n} (1+\abs{p}^2)^{-(n+2)/2}
 \]
is necessary for the uniqueness of the viscosity solution.  To see why, we return to our earlier one-dimensional example, posed on a larger domain.

\begin{examp}
\bq\label{eq:1d_2}
F(x,u,u_x,u_{xx}) 
=
\begin{cases}
-u_{xx} + (1+u_x^2)^{3/2}, & x\in (-1,1)\\
u, & x  = \pm 1.\\
\end{cases}
\eq
This time, the inequality is not strict:
\[ \int_{-1}^1\kappa(x)\,dx = 2 = \int_{-\infty}^\infty (1+p^2)^{-3/2}\,dp.\]
We claim that for any $a \geq 0$, the function
\[u(x) = -\sqrt{1-x^2}-a \]
is a viscosity solution.  As before, this is a classical solution in the interior $(-1,1)$.  On the boundary, $u(x) \leq 0$ so it is also a sub-solution.  Finally, we note that $u'(x)$ becomes unbounded at the boundary and it is therefore impossible to place any smooth test function $\phi$ below $u$ at $\pm1$.  Thus the super-solution condition is trivially satisfied at the boundary.  We conclude that strict inequality in the condition of Lemma~\ref{lem:exist} must be needed to guarantee uniqueness of the viscosity solution. \hfill $\square$
\end{examp}

An important property of many elliptic equations is the comparison principle, which immediately implies uniqueness of the solution.
\begin{definition}[Comparison principle]\label{def:comparison}
A PDE has a \emph{comparison principle} if whenever $u$ is an upper semi-continuous sub-solution and $v$ a lower semi-continuous super-solution of the equation, then $u \leq v$ on $\bar{\Omega}$.
\end{definition}

The comparison principle plays an important role in developing convergent approximation schemes using the Barles-Souganidis framework~\cite{BSnum}.  As we shall see, our equation equipped with a weak Dirichlet condition does not satisfy a comparison principle in the traditional sense.  

\begin{examp}
To see why this must be the case, we return to the one-dimensional example considered in section~\ref{sec:bc}.  We have already seen that the function $u = -\sqrt{1-x^2}$ is a viscosity solution, and therefore a viscosity super-solution as well.  

Now we claim that the upper semi-continuous function
\[
v(x) = \begin{cases}
-\sqrt{1-x^2}, & x\in[0,1)\\
1, & x = 1
\end{cases}
\]
is a sub-solution; see Figure~\ref{fig:sub1D}.  As before, the appropriate conditions are trivially satisfied for $x\in(0,1]$ and we need only check $x = 1$.  Now any test function $\phi$ will satisfy
\[ F_*(1,v(1),\phi_x(1),\phi_{xx}(1)) \leq v(1) - 1 = 0. \]
We conclude that $v$ is a sub-solution, $u$ a super-solution, and $v(1)>u(1)$.  Thus this equation cannot satisfy a comparison principle in the sense of Definition~\ref{def:comparison}. \hfill $\square$
\end{examp}

Instead, we will develop a relaxed comparison principle, which will be used to produce a convergence proof via a modification of the usual framework.

Proofs of comparison principles are not available for general elliptic PDEs and often rely on particular details of the structure of a given PDE operator.  While the prescribed Gaussian curvature equation does not satisfy the structure condition typically used to prove comparison~\cite{CIL}, Ishii and Lions have shown comparison for a class of \MA equations that satisfy a much weaker structure condition~\cite[Theorem~V.2]{IshiiLions}.
%
An immediate consequence of this result is a comparison principle for our equation of interest.
\begin{thm}[Classical comparison principle for Gaussian curvature]\label{thm:comparisonGauss}
Suppose\newline $\kappa:\Omega\to[0,\infty)$ is continuous and bounded and let $u$, $v$ be respectively sub- and super-solutions of the PDE for prescribed Gaussian curvature~\eqref{eq:MA}.  If $u \leq v$ on $\partial\Omega$ then $u \leq v$ on $\bar{\Omega}$.
\end{thm}

This yields a uniqueness result for the \MA equation if a solution exists that satisfies the Dirichlet boundary conditions in a classical sense.  However, if we want to interpret the boundary conditions in the weak sense of~\eqref{eq:weakBC}, the comparison principle only applies to continuous functions and not to general semi-continuous sub- and super-solutions~\cite{CIL}.  Modified comparison principles for non-continuous solutions have been proved for Hamilton-Jacobi equations by exploiting the control interpretation of the problem, but this approach does not apply to our setting~\cite{BardiDolcetta}. Instead, we will use the geometric interpretation of the generalised solution to demonstrate uniqueness of the viscosity solution, then use this to prove a modified comparison principle that is valid only in the interior of the domain.  In section~\ref{sec:approx}, this comparison principle will be used to prove the convergence of appropriate approximations to the solution of~\eqref{eq:MA}.

\begin{thm}[Interior comparison principle for weak Dirichlet problem]\label{thm:comparison}
Assume\newline Hypothesis~\ref{hyp} holds.
If $u$ is a bounded upper semi-continuous sub-solution and $v$ a bounded lower semi-continuous super-solution of~\eqref{eq:MA} then  $u \leq v$ on ${\Omega}$.
\end{thm}

The proof of this theorem is deferred until the end of this section.  We first build up some necessary machinery.
A key ingredient is the fact that the sub-gradients of two ordered functions must themselves be ordered under appropriate boundary conditions.  This result is known when the two functions are identical at the boundary~\cite[Lemma~1.4.1]{Gutierrez}.  We prove a similar result under a much weaker condition, requiring that at all boundary points either the functions are equal or the gradient of the lower function is unbounded.  As this result is used in the proof of Theorem~\ref{thm:equivalence1}, we emphasise that it is a consequence of the definition of the subgradient and does not depend on any other results of this article.

\begin{lem}[Subgradient sets are ordered]\label{lem:subgradOrder}
Let $u,v$ be  lower semi-continuous and $u \leq v$ on an open set $E$.  Suppose also that at each boundary point $x_0\in\partial E$ either $v(x_0)=u(x_0)$ or $\partial u(x_0)$ is empty.  Then $\partial v(E) \subset \partial u(E)$.
\end{lem}

\begin{proof}
Choose any $x_0\in E$ and $p\in\partial v(x_0)$ and define the value
\bq\label{eq:supp} a \equiv \sup\limits_{x\in E}\{v(x_0) + p\cdot(x-x_0) - u(x)\} \geq 0. \eq
We claim that
\[ v(x_0) + p\cdot(x-x_0) - a \]
is a supporting hyperplane to $u$.

Since $u$ is lower semi-continuous, there exists some $x_1\in\bar{E}$ such that
\[ a = v(x_0) + p\cdot(x_1-x_0)-u(x_1). \]
This enables us to rewrite the definition of~$a$ in~\eqref{eq:supp} as
\[ u(x) \geq v(x_0) + p\cdot (x-x_0)-a = u(x_1) + p\cdot(x-x_1) \]
for every $x\in E$.  Therefore $p\in\partial u(x_1)$.  We still need to demonstrate that $x_1\in E$ is an interior point.

We recall that since $p\in\partial v(x_0)$,
\[ v(x_1) \geq v(x_0) + p\cdot (x_1-x_0) = u(x_1) + a. \]

{\bf Case 1}: $a > 0$. 
Then $v(x_1) > u(x_1)$.  Suppose that $x_1\in\partial E$.  Since $u(x_1) \neq v(x_1)$ it must be the case that $\partial u(x_1)$ is empty, which contradicts the fact that $p\in\partial u(x_1)$.  We conclude that $x_1\in E$ and $p\in\partial u(E)$.

{\bf Case 2}: $a = 0$.
Then for every $x\in E$,
\[ u(x) \geq v(x_0) + p\cdot (x-x_0) \geq u(x_0) + p\cdot(x-x_0) \]
and $p \in \partial u(x_0)\subset\partial u(E)$. 
\end{proof}

Using this ordering of subgradients, we now demonstrate that the viscosity solution is unique.

\begin{thm}[Viscosity solution is unique]\label{thm:uniqueness}
Assume Hypothesis~\ref{hyp} holds. Let~$u$ be the maximal sub-solution of~\eqref{eq:MA} and let~$v$ be any viscosity solution.  Then $u=v$ on $\Omega$.
\end{thm}

\begin{proof}
By Theorem~\ref{thm:exist}, $u$ is a viscosity solution.
Note that $v \leq u$ necessarily since $u$ is maximal.  Additionally, at all boundary points $x_0\in\partial\Omega$, either $\partial v_*(x_0)$ is empty or $v_*(x_0) = v^*(x_0) = u_*(x_0) = u^*(x_0) = g(x_0)$; see Corollary~\ref{cor:viscBC}.

Choose any $x_0\in\Omega$ and consider the function 
\[ w(x) = u(x)-u(x_0)+v(x_0) \leq u(x). \]
Notice that $\partial w(x) = \partial u(x)$.
Now define the set 
\[ E = \{x\in\Omega\mid w(x) \geq v(x)\}. \]


{\bf Case 1}: $x_0\notin\partial E$.  Then $w-v$ has a minimum at $x_0$ or is constant nearby and $\partial v(x_0) \subset \partial w(x_0)$.

{\bf Case 2}: $x_0\in\partial E$. 
Now for any $z\in\partial E$ either
  $w_*(z) = v_*(z)$ or 
  $w_*(z)>v_*(z)$ with $z\in\partial\Omega$.  In the latter case, we must have $v_*(z)<u_*(z) \leq u^*(z) \leq g(z)$ so that $\partial v_*(z)$ is empty (Lemma~\ref{lem:superBC}).  Thus the hypotheses of Lemma~\ref{lem:subgradOrder} are satisfied and $\partial w(E) \subset \partial v(E)$.

We also note that both $w$ and $v$ are generalised solutions of the prescribed Gaussian curvature equation, which means that
\[ \int_{\partial v(E)}(1+\abs{p}^2)^{-(n+2)/2}\,dp = \int_{\partial w(E)}(1+\abs{p}^2)^{-(n+2)/2}\,dp. \]
We conclude that the boundaries of the sets $\partial v(E)$, $\partial w(E)$ must be identical. 

Now let $p\in\partial w(x_0)$ with  $p$ close to subgradient values coming from the interior of the set $E$.  That is, there exists a sequence $x_n \in E$, $p_n\in\partial w(x_n)$ such that $x_n\to x_0$ and $p_n\to p$.   Then there must be another boundary point $z\in\partial E$ such that $p\in\partial v(z)$ and for every $x\in\Omega$:
\begin{align*}
w(x) &\geq w(x_0) + p\cdot(x-x_0)\\
v(x) &\geq v(z) + p\cdot(x-z).
\end{align*}
A consequence of this is that for every $x\notin E$,
\[ v(x) > w(x) \geq w(x_0) + p\cdot(x-x_0) = v(x_0) + p\cdot(x-x_0). \]

We suppose that $p\notin\partial v(x_0)$ and seek a contradiction.  In particular, this means that there exists some $y\in E$ such that
\[ v(y) < v(x_0) + p\cdot(y-x_0) = w(x_0) + p\cdot(y-x_0). \]
Combining these inequalities, we obtain
\begin{align*}
w(x_0) + p\cdot(y-x_0) > v(y) \geq v(z) + p\cdot(y-z) = w(z) + p\cdot(y-z).
\end{align*}
Rearranging this yields
\[ w(z) < w(x_0) + p\cdot(z-x_0), \]
which contradicts the definition of $p\in\partial w(x_0)$.  Thus we must have $p\in\partial v(x_0)$.

We conclude that at all points $x\in\Omega$, the intersection  $\partial v(x)\cap\partial w(x)$ is non-empty and thus $\partial v(x) = \partial w(x) = \partial u(x)$.  Therefore $u(x)-v(x) = c$ is constant.

Since $u$ is the maximal solution, $c \geq 0$.  If $c > 0$, then $v_*(x) < u_*(x) \leq g(x)$ at all points on the boundary $\partial\Omega$.  From Lemma~\ref{lem:superBC}, $\partial v_*(x)$ must be empty at all points on the boundary.  Then since $v$ is convex, we must have $\partial v(\Omega) = \R^n$.  Since $v$ is a generalised solution of the prescribed Gaussian curvature equation we can compute
\begin{align*} \int_{\R^n} (1+\abs{p}^2)^{-(n+2)/2}\,dp &= \int_{\partial v(\Omega)}(1+\abs{p}^2)^{-(n+2)/2}\,dp\\ &= \int_\Omega \kappa(x)\,dx < \int_{\R^n} (1+\abs{p}^2)^{-(n+2)/2}\,dp.  \end{align*}
This is a contradiction, which means that $c=0$ and $v = u$.
\end{proof}

Now we are able to establish the interior comparison principle for general semi-continuous functions.

%
%

\begin{proof}[Proof of Theorem~\ref{thm:comparison}]
Since $u,v$ are bounded, for sufficiently large $M>0$ we have
\[ u-M \leq {v}. \]
As $u$ is a sub-solution, $u-M$ is as well.

By Theorems~\ref{thm:perron} and~\ref{thm:uniqueness}, the unique viscosity solution of~\eqref{eq:MA} can be expressed as
\[ w(x) =  \sup\{W(x) \mid u - M \leq W \leq {v}, \, W^* \text{ is a sub-solution}\}.\]

By Corollary~\ref{cor:maxSub}, $w$ is the maximal sub-solution so that $u \leq w$.  Thus this characterisation implies that in~$\Omega$,
%
\[ u \leq w \leq  v.  \]
\end{proof}

\section{Convergence of Elliptic Schemes}\label{sec:convergeElliptic}

With a solid theoretical understanding of the generalised Dirichlet problem in place, we now turn our attention to developing criteria that should be satisfied by a convergent numerical method.  We begin by developing a framework that applies to general elliptic PDEs with an interior comparison principle, under the mild condition that it is possible to construct strict classical sub- and super-solutions.  In section~\ref{sec:approx}, we will use this framework to develop and analyse a convergent numerical method for the equation of prescribed Gaussian curvature.

\subsection{Properties of schemes}\label{sec:properties}

Consider a set of discretisation points $\G^h\subset\bar{\Omega}$, which can be a finite difference grid or a more general point cloud.  Here $h$ is a small parameter relating to the grid resolution. In particular, we expect that as $h\to0$, the domain becomes fully resolved in the sense that
\bq\label{eq:resolution} \lim\limits_{h\to0} \sup\limits_{y\in\Omega}\min\limits_{x\in\G^h}\abs{x-y} = 0. \eq

To produce consistent, monotone approximations near the boundary, we will later require a sufficiently high boundary resolution $h_B$,
\bq\label{eq:bdyRes} h_B = \sup\limits_{y\in\partial\Omega}\min\limits_{x\in\G^h\cap\partial\Omega}\abs{x-y} \eq
with $h_B \ll h$.

We consider finite difference schemes that have the form
\bq\label{eq:approx} F^h(x,u(x),u(x)-u(\cdot)) = 0, \quad x\in \G^h \eq
where $u:\G^h\to\R$ is a grid function.


Our convergence framework requires schemes to be consistent, monotone, and Lipschitz continuous.

\begin{definition}[Consistency]\label{def:consistency}
The scheme~\eqref{eq:approx} is \emph{consistent} with the PDE
\bq\label{eq:PDEGen} F(x,u(x),\nabla u(x), D^2u(x)) = 0, \quad x\in\bar{\Omega} \eq
 if for any smooth function $\phi$ and $x\in\bar{\Omega}$,
\[ \limsup_{h\to0,y\to x, z\in\G^h\to x,\xi\to0} F^h(z,\phi(y)+\xi,\phi(y)-\phi(\cdot)) \leq F^*(x,\phi(x),\nabla\phi(x),D^2\phi(x)), 
\]
\[ \liminf_{h\to0,y\to x, z\in\G^h\to x,\xi\to0} F^h(y,\phi(y)+\xi,\phi(y)-\phi(\cdot)) \geq F_*(x,\phi(x),\nabla\phi(x),D^2\phi(x)). \]
\end{definition}

\begin{definition}[Monotonicity]\label{def:monotonicity}
The scheme~\eqref{eq:approx} is \emph{monotone} if $F^h$ is a non-decreasing function of its final two arguments.
\end{definition}

\begin{definition}[Lipschitz]\label{def:lipschitz}
The scheme~\eqref{eq:approx} is \emph{Lipschitz} if $F^h$ is locally Lipschitz continuous in its final two arguments. 
\end{definition}

These properties ensure that the approximation schemes inherit the basic structure of the underlying elliptic PDE.  In particular,
we note that monotone (elliptic) finite difference schemes enjoy a discrete comparison principle.

\begin{lem}[Discrete comparison principle~{\cite[Theorem~5]{ObermanSINUM}}]\label{lem:discreteComp}
Let $F^h$ be a monotone scheme and $F^h(x,u(x),u(x)-u(\cdot)) < F^h(x,v(x),v(x)-v(\cdot))$ for every $x\in\G^h$.  Then $u(x) \leq v(x)$ for every $x\in\G^h$.
\end{lem}

\begin{rem}
Because the inequality in this discrete comparison principle is strict, it does not guarantee solution uniqueness.  For some monotone schemes, it is not possible to find $u, v$ such that $F^h[u] < F^h[v]$ at every grid point.
\end{rem}

\subsection{Existence and stability}\label{sec:wellPosed}

For elliptic finite difference schemes that are also \emph{proper} (i.e. $F^h$ is a strictly increasing function of its second argument) and globally Lipschitz continuous, the results of~\cite{ObermanSINUM} establish that solutions exist and are stable.  In this section, we prove similar results for schemes that do not satisfy these extra conditions.  We will instead require a mild assumption on the underlying PDE, which must have a strict classical sub- and super-solution.  A bounded solution of the scheme can then be constructed using a discrete verion of Perron's method.

\begin{definition}[Strict classical sub(super)-solution]\label{def:strict}
A function $u\in C^2$ is a strict classical sub(super)-solution of the PDE~\eqref{eq:PDEGen} if there exists some $\mu>0$ such that
\[ F^*(x,u(x),\nabla u(x), D^2u(x)) \leq -\mu \quad \left(F_*(x,u(x),\nabla u(x), D^2u(x)) \geq \mu\right)\]
for every $x\in\bar{\Omega}$.
\end{definition}

\begin{rem}
For many elliptic PDEs, strict sub- and super-solutions can be obtained from simple quadratic functions.  It is slightly more involved in our setting because the right-hand side of~\eqref{eq:MA} can be unbounded.  However, we will demonstrate in section~\ref{sec:approx} that standard elliptic theory can be used to construct appropriate sub- and super-solutions.
\end{rem}

\begin{lem}[Existence]\label{lem:existDiscrete}
Let $F^h$ be a consistent, monotone, Lipschitz scheme.  Suppose also that there exist functions $v, w\in C^2(\bar{\Omega})$ such that $v$ is a strict sub-solution and $w$ a strict super-solution of the underlying PDE.  Then for sufficiently small $h>0$, the approximation scheme~\eqref{eq:approx} has a solution.
\end{lem}

\begin{proof}
First we note that by consistency, we can restrict $v,w$ to the grid and obtain
\[ F^h(x,v(x),v(x)-v(\cdot)) < 0, \quad F^h(x,w(x),w(x)-w(\cdot))>0 \]
for any $x\in\G^h$ and sufficiently small $h$.  By the discrete comparison principle (Lemma~\ref{lem:discreteComp}), $v \leq w$.

Now define the grid function 
\bq\label{eq:discretePerron} u = \sup\left\{U \mid U(x) \leq w(x), \, F^h(x,U(x),U(x)-U(\cdot)) < 0, \,\forall x\in\G^h \right\}, \eq
which is well-defined since $v$ satisfies both of the constraints.  We claim that $u$ is a solution of~\eqref{eq:approx}.

Consider any $x\in\G^h$ and $\epsilon>0$.  From the definition of $u$, there exists a strict discrete sub-solution $u^\epsilon$ such that
\[u^\epsilon(x) > u(x)-\epsilon, \quad u^\epsilon(\cdot) \leq u(\cdot).\]
Then we can use the monotonicity of the scheme to compute
\begin{align*}
0 & > F^h(x,u^\epsilon(x),u^\epsilon(x)-u^\epsilon(\cdot))\\
  &\geq F^h(x,u(x)-\epsilon,u(x)-\epsilon-u^\epsilon(\cdot))\\
	&\geq F^h(x,u(x)-\epsilon,u(x)-\epsilon-u(\cdot)).
\end{align*}
Since $F^h$ is Lipschitz, we can take $\epsilon\to0$ to obtain
\[ F^h(x,u(x),u(x)-u(\cdot)) \leq 0 \]
and thus $u$ is a sub-solution of the scheme.

Next we suppose that there is some $y\in\G^h$ such that $F(y,u(y),u(y)-u(\cdot)) < 0$.  We will show that we can construct a larger sub-solution.  Choose $\epsilon>0$ and consider
\[ \tilde{u}(x) = \begin{cases} u(x), & x\neq y\\u(x)+\epsilon, & x=y. \end{cases} \]
We will verify that this is a sub-solution of the scheme.  First consider $x \neq y$.  Since $\tilde{u}(x)=u(x)$ and $\tilde{u}(y)>u(y)$, monotonicity of the scheme yields
\[ F^h(x,\tilde{u}(x),\tilde{u}(x)-\tilde{u}(\cdot)) \leq F^h(x,u(x),u(x)-u(\cdot)) \leq 0.  \]
Additionally, for small enough $\epsilon>0$, the Lipschitz continuity of the scheme yields
\[ F^h(y,\tilde{u}(y),\tilde{u}(y)-\tilde{u}(\cdot)) = F^h(y,u(y)+\epsilon,u(y)+\epsilon-u(\cdot)) < 0. \]

Thus $u^\epsilon$ is a sub-solution of the scheme with $\max\{u^\epsilon - u\} = \epsilon>0$.
This contradicts the definition of $u$ as the maximal sub-solution in~\eqref{eq:discretePerron} and we conclude that
\[ F^h(x,u(x),u(x)-u(\cdot)) = 0, \quad x\in \G^h. \qedhere \]
\end{proof}

As a simple consequence of the discrete comparison principle, we can also obtain bounds on the solution of the scheme.

\begin{lem}[Stability]\label{lem:stability}
Let $F^h$ be a consistent, monotone, Lipschitz scheme and let $u^h$ be a solution of~\eqref{eq:approx}.  Suppose also that there exist functions $v, w\in C^2(\bar{\Omega})$ such that $v$ is a strict sub-solution and $w$ a strict super-solution of the underlying PDE.  Then there exists a constant $M>0$, independent of $h$, such that $\|u^h\|_\infty \leq M$ for sufficiently small $h>0$.
\end{lem}

\begin{proof}
As in the previous lemma, $v$ and $w$ are strict sub- and super-solutions of the scheme for small enough $h>0$.  By the discrete comparison principle (Lemma~\ref{lem:discreteComp}), we have $v \leq u^h \leq w$ and thus $\|u^h\|_\infty \leq \max\{\|v\|_\infty,\|w\|_\infty\}$.
\end{proof}

\subsection{Convergence}\label{sec:convergence}

The concepts of consistency, monotonicity, stability, and interior comparison can now be used to prove that elliptic approximation schemes converge to the viscosity solution of the underlying PDE.  This is accomplished through a slight modification of the well-known Barles-Souganidis convergence framework.  While the proof below is similar to those in~\cite{BSnum,FOFiltered}, those works implicitly required  the approximation scheme to be defined throughout the domain.  

Here we are interested in schemes that are defined only on a finite set of discretisation points.  In order to modify the convergence proof accordingly, we need to extend the discrete grid solution into the entire domain $\bar{\Omega}$.  To this end, we let $U^h:\G^h\to\R$ be a solution of the approximation scheme on the grid.  Using this, we define the piecewise constant extension
\bq\label{eq:extension} u^h(x) = \max\left\{U^h(y) \mid y\in\G^h, \, \abs{y-x} = \min\limits_{z\in\G^h}\abs{z-x}\right\}. \eq
This is simply a nearest neighbours extension, which accounts for the situation where multiple discretisation points are equidistant.

\begin{thm}[Convergence of Schemes]\label{thm:convergence}
Consider a degenerate elliptic PDE~\eqref{eq:PDEGen} on a bounded domain~$\Omega$.  Suppose that the PDE operator satisfies an interior comparison principle and that there exist strict classical sub- and super-solutions to the PDE.  Let $F^h$ be any consistent, monotone, Lipschitz scheme and  $U^h$ any solution of the scheme.  Then for any interior point $x\in\Omega$, the piecewise constant extension $u^h(x)$ converges to the viscosity solution of the underlying PDE as $h\to0$.
\end{thm}

\begin{rem}
The above theory provides existence but not necessarily uniqueness of solutions $U^h$ to the scheme $F^h = 0$.  However, all solutions converge in the limit as $h\to0$.
\end{rem}

\begin{rem}
As in~\cite{FOFiltered}, it is sufficient to use a perturbation of a monotone scheme, which allows for the construction of convergent, formally higher-order approximations.
\end{rem}

\begin{rem}
If the PDE satisfies a comparison principle in the closure of the domain, then the scheme converges in $\bar{\Omega}$.
\end{rem}

\begin{proof}[Proof of Theorem~\ref{thm:convergence}]
The key to the proof is to extend the approximation operator $F^h$ from the grid onto the entire domain.  

Define
\[ \bar{u}(x) = \limsup_{h\to0,y\to x} u^h(y) \in USC(\bar{\Omega}), \]
\[ \underline{u}(x) = \liminf_{h\to0,y\to x} u^h(y) \in LSC(\bar{\Omega}). \]
Clearly $\underline{u}(x) \leq  \bar{u}(x)$ everywhere in $\bar{\Omega}$.

From Lemma~\ref{lem:stability}, both $\bar{u}$ and $\underline{u}$ are bounded.  

Consider any $x_0\in\bar{\Omega}$ and $\phi\in C^2$ such that $x_0$ is a strict global maximum of $\bar{u}-\phi$ with $\bar{u}(x_0)=\phi(x_0)$.
Maxima of upper semi-continuous functions are stable and thus it is possible to find subsequences 
\[ h_n\to0, \quad y_n \to x_0, \quad u^{h_n}(y_n)\to\bar{u}(x_0) \]
where $y_n$ is a maximiser of $u^{h_n}-\phi$.  
See, for example~\cite[Lemma 2]{FOFiltered}.

Since $u^h$ is a nearest neighbours extension of the grid solution and $h$ measures the resolution of the underlying grid or point cloud~\eqref{eq:resolution}, we can also find $z_n\in\G^h$ such that $z_n\to x_0$ and $u^{h_n}(z_n) = u^{h_n}(y_n)$.  Defining $\xi_n = u^{h_n}(y_n)-\phi(y_n)\to0$, we have
\[
u^{h_n}(y_n) = \bar{u}(x_0) + \xi_n
\]
From the definition of the various subsequences, we also obtain
\[ u^{h_n}(z_n)-u^{h_n}(\cdot) = u^{h_n}(y_n)-u^{h_n}(\cdot) \geq \phi(y_n)-\phi(\cdot). \]

Then with a slight modification of~\cite{BSnum},  we can use monotonicity to verify that
\begin{align*}
0 &= F^{h_n}(z_n,u^{h_n}(z_n),u^{h_n}(z_n)-u^{h_n}(\cdot))\\
  &\geq F^{h_n}(z_n,\phi(y_n)+\xi_n,\phi(y_n)-\phi(\cdot)).
\end{align*}

From consistency of the approximation, we obtain
\begin{align*}
0 &\geq \liminf\limits_{n\to\infty}F^{h_n}(z_n,\phi(y_n)+\xi_n,\phi(y_n)-\phi(\cdot))\\
  &\geq \liminf\limits_{h\to0,y\to x_0, z\in\G^h\to x_0,\xi\to0}F^{h_n}(z,\phi(y)+\xi,\phi(y)-\phi(\cdot))\\
	&\geq F_*(x_0,\phi(x_0),\nabla\phi(x_0),D^2\phi(x_0)).
\end{align*}
Thus $\bar{u}$ is a sub-solution of the PDE.  We can similarly show that $\underline{u}$ a super-solution.

If $u$ is the viscosity solution of the PDE then $u^*$ is a sub-solution and $u_*$ is a super-solution.  For $x\in\Omega$, two applications of the comparison principle yields
\[ u(x) \leq u^*(x) \leq \underline{u}(x) \leq \bar{u}(x) \leq u_*(x) \leq u(x). \]
We conclude that $u = \bar{u} = \underline{u}$ in $\Omega$ and therefore
$u^h(x)$ converges to the viscosity solution at interior points $x\in\Omega$. 
\end{proof}

An immediate consequence of this result is convergence in $L^p$.
\begin{cor}[Convergence in $L^p$]\label{cor:Lp}
Under the hypotheses of Theorem~\ref{thm:filtered}, $u^h$ converges to $u$ in $L^p$ for any $0<p<\infty$.
\end{cor}

\begin{proof}
Choose any $\epsilon>0$ and let $\Omega_\epsilon$ be an $\epsilon$-neighbourhood of the boundary $\partial\Omega$:
\[ \Omega_\epsilon = \{x\in\Omega \mid \text{dist}(x,\partial\Omega)<\epsilon\}. \]
Since $\|u^h\|_\infty < M$ (Lemma~\ref{lem:stability}) and $\Omega$ is bounded, we can bound the $L^p$ error by
\begin{align*}
\lim_{h\to0}\|u-u^h\|_p^p &= \lim_{h\to0}\left(\int_{\Omega\backslash\Omega_{\epsilon}}\abs{u(x)-u^h(x)}^p\,dx + \int_{\Omega_\epsilon}\abs{u(x)-u^h(x)}^p\,dx\right)\\
  &\leq \abs{\Omega\backslash\Omega_{\epsilon}}\lim_{h\to0}\sup\limits_{\Omega\backslash\Omega_{\epsilon}}\abs{u(x)-u^h(x)}^p + \left(M+\sup\abs{u}\right)^p\abs{\Omega_\epsilon}\\
	&= \left(M+\sup\abs{u}\right)^p\abs{\Omega_\epsilon}.
\end{align*}
Since $\Omega$ is bounded, we can take $\epsilon\to0$ to obtain
\[ \lim_{h\to0}\|u-u^h\|_p^p = 0. \qedhere \]
\end{proof}

\section{Numerical Method for Prescribed Gaussian Curvature}\label{sec:approx}

Now we can use the results of the previous sections to produce a provably convergent method for computing generalised solutions of the prescribed Gaussian curvature equation.  For simplicity and brevity, we will describe the scheme in 2D, although the techniques and proofs can be adapted to higher dimensions.

\subsection{Discretisation in interior}\label{sec:discInterior}
We begin by reviewing the techniques needed to produce a monotone discretisation of the equation.

At interior points, we can rely on a slight modification of monotone schemes that have previously been proposed for the solution of \MA equations.  
We briefly describe the monotone scheme that we use for the prescribed Gaussian curvature equation, and refer to~\cite{FroeseMeshfreeEigs,FO_MATheory} for further details.  This requires constructing monotone approximations for the terms
\[ -{\det}^+(D^2u(x)), \quad \kappa(x)R(\nabla u(x)), \quad -\lambda_1(D^2u(x)). \]

Monotone approximations for the determinant of the Hessian have been thoroughly described in~\cite{FO_MATheory}.  Briefly, this approximation is based on the characterisation
\[ -{\det}^+(D^2u) =
-\min\limits_{\{\nu_1\ldots\nu_n\}\in V} 
\prod\limits_{i=1}^{n} 
\max\{u_{\nu_i\nu_i} ,0\}
 \]
where $V$ is the set of all orthogonal coordinate systems in $\R^n$ and $u_{\nu\nu}$ is the second directional derivative of $u$ in the direction $\nu$.

Instead of considering all orthogonal coordinate systems, we consider a finite subset $V^h$ of $V$, which necessarily introduces some angular resolution error $d\theta$ into the scheme. In our implementation, we consider the subset
\[ V^h = \{(\cos(j\,d\theta),\sin(j\,d\theta)), (-\sin(j\,d\theta),\cos(j\,d\theta))\}, \quad j = 1, \ldots, \frac{\pi}{2d\theta} \]
where we take $d\theta \approx 2\pi h^{1/4}$.

If the direction $\nu$ is of the form $x_j-x_i$ where $x_i$ and $x_j$ are different grid points on a Cartesian grid, then
the second derivatives can be discretised using centred differences.
\[u_{\nu\nu}(x_0) \approx \frac{1}{\abs{\nu}^2}\left(u(x_0 + \nu) + u(x_0 - \nu) - 2u(x_0)\right).\]
For more general directions or grids, the scheme can be modified as in~\cite{FroeseMeshfreeEigs}. To accomplish this, we 
consider a search neighbourhood of radius
\[ \delta = h(1+\cos(d\theta/2)\cot(d\theta/2)+\sin(d\theta/2)). \]
Neighbouring grid points can be written in polar coordinates $(r,\phi)$ with respect to the axes defined by the lines $x_0 + t\nu$, $x_0 + t\nu^\perp$.  We seek one neighbouring discretisation point in each quadrant described by these axes, with each neighbour aligning as closely as possible with the line $x_0 + t\nu$.  That is, we select the neighbours
\[ x_j \in \argmin\left\{{\sin^2\phi} \mid (r,\phi)\in\G^h\cap B(x_0,\delta) \text{ is in the $j$th quadrant}\right\}\]
for $j = 1, \ldots, 4$.  See Figure~\ref{fig:stencil}.
Because of the ``wide-stencil'' nature of these approximations (since the search radius $\delta \gg h$), care must be taken near the boundary.  In order to preserve consistency up to the boundary, it is necessary that the boundary be more highly resolved than the interior ($h_B \ll h$).

\begin{figure}[htp]
\centering
\subfigure[]{
\includegraphics[width=0.4\textwidth]{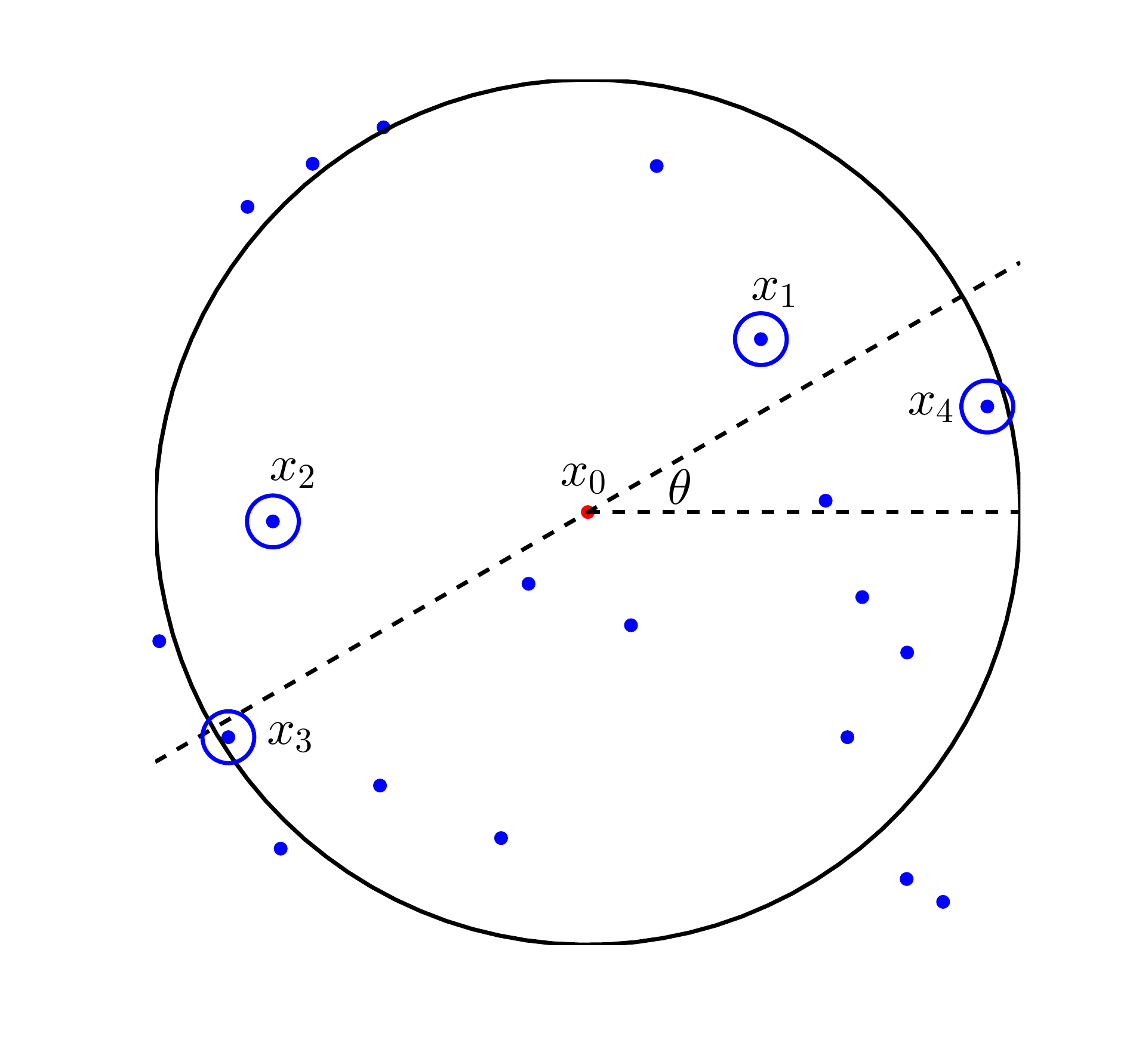}\label{fig:stencil1}}
\subfigure[]{
\includegraphics[width=0.47\textwidth]{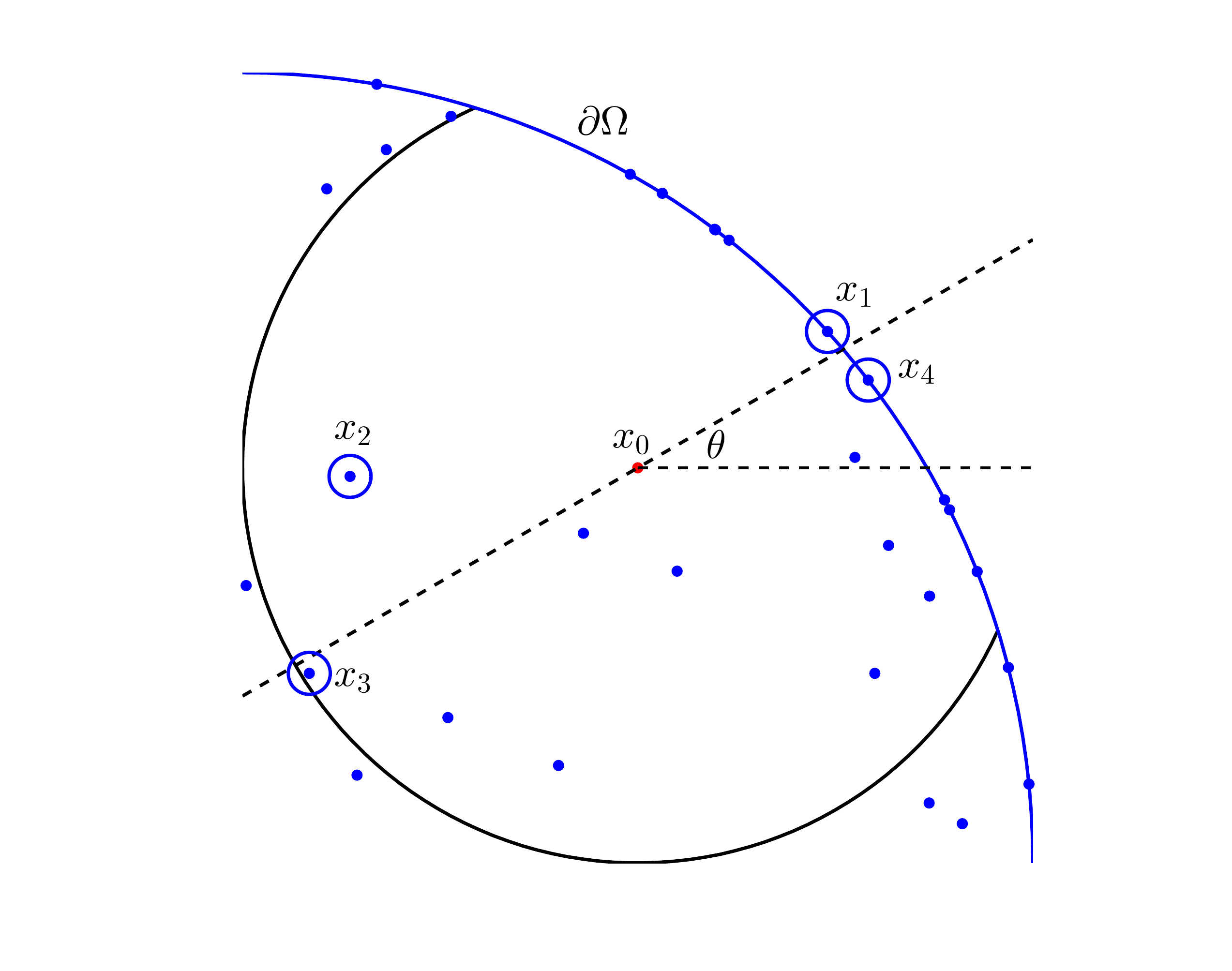}\label{fig:stencil2}}
\caption{A finite difference stencil chosen from a point cloud \subref{fig:stencil1}~in the interior and \subref{fig:stencil2}~near the boundary.}
\label{fig:stencil}
\end{figure}

Then a consistent, monotone approximation of $u_{\nu\nu}$ is
\[ \Dt_{\nu\nu}^hu(x_0) = \sum\limits_{j=1}^4 a_j(u(x_j)-u(x_0)) \]
where we use the polar coordinate characterisation of the neighbours to define
\[ S_j =  r_j\sin\phi_j, \quad C_j = r_j\cos\phi_j\]
and the coefficients are given by
\[\begin{split}
a_1 &= \frac{2S_4(C_3S_2-C_2S_3)}{(C_3S_2-C_2S_3)(C_1^2S_4-C_4^2S_1)-(C_1S_4-C_4S_1)(C_3^2S_2-C_2^2S_3)}\\
a_2 &= \frac{2S_3(C_1S_4-C_4S_1)}{(C_3S_2-C_2S_3)(C_1^2S_4-C_4^2S_1)-(C_1S_4-C_4S_1)(C_3^2S_2-C_2^2S_3)}\\
a_3 &= \frac{-2S_2(C_1S_4-C_4S_1)}{(C_3S_2-C_2S_3)(C_1^2S_4-C_4^2S_1)-(C_1S_4-C_4S_1)(C_3^2S_2-C_2^2S_3)}\\
a_4 &= \frac{-2S_1(C_3S_2-C_2S_3)}{(C_3S_2-C_2S_3)(C_1^2S_4-C_4^2S_1)-(C_1S_4-C_4S_1)(C_3^2S_2-C_2^2S_3)}.
\end{split}
\]
Using this, we define the following discrete approximation to the first term in the convexified \MA operator:
\[ {\det}^h(D^2u(x_0)) = \min\limits_{\{\nu_1,\nu_2\}\in V^h} \prod\limits_{i=1}^{2} 
\max\{\Dt^h_{\nu_i\nu_i}u(x_0) ,0\}
 \]

Next, we consider the term
\[ \kappa(x)R(\nabla u(x)) = \kappa(x) (1 + \abs{\nabla u(x)}^2)^{(n+2)/2}. \]
Since the curvature $\kappa$ is non-negative, it is only necessary to consider a monotone discretisation of the gradient term.  In this case, we can make use of the structure of the function $R$ and use a monotone scheme for $\abs{\nabla u}$ that has previously been used to solve the Eikonal equation~\cite{Zhao}.
When the grid is uniform, the approximation is
\[
\abs{\nabla u(x_i)}^2 \approx \sum\limits_{i=1}^n \max\left\{\frac{u(x)-u(x+he_i)}{h},\frac{u(x)-u(x-he_i)}{h}, 0\right\}^2.
\]
More generally, we can consider the four neighbours $x_j^i$, $j = 1, \ldots, 4$ that best align with the coordinate direction $e_i$.  Using the above notation, 
\[
\begin{split}
\abs{\Dt^h u(x_0)}^2 = \sum\limits_{i=1}^2 \max\left\{b_1^i(u(x_1^i)-u(x_0))+b_4^i(u(x_4^i)-u(x_0)), \right.\\ \left. b_2^i(u(x_2^i)-u(x_0))+b_3^i(u(x_3^i)-u(x_0)), 0\right\}^2
\end{split}\]
where the cofficients are given by

\begin{align*}
b_1= \frac{S_4}{S_1C_4-C_1S_4} \phantom{phantom}& 
b_2 = -\frac{S_3}{S_2C_3-C_2S_3}\\
b_3 = \frac{S_2}{S_2C_3-C_2S_3} \phantom{phantom} &
b_4 = -\frac{S_1}{S_1C_4-C_1S_4}.
\end{align*}

Finally, we discretise the term
\[ -\lambda_1(D^2u). \]
Following Oberman~\cite{ObermanEigenvalues}, we can rewrite the smallest eigenvalue as
\[ \lambda_1(D^2u) = \min\limits_{\abs{\nu} = 1} u_{\nu\nu}. \]
As with the \MA equation, we can approximate this using a finite set of directions.  
\[ \lambda_1^h(D^2u(x_0)) = \min \left\{\Dt^h_{\nu_i\nu_i}u(x_0) \mid \nu_i = i\,d\theta, \, i = 1, \ldots, \frac{\pi}{d\theta}\right\}. \]

Finally, we can define the overall approximation scheme at interior points $x\in\G^h\cap\Omega$ by
\bq\label{eq:approxInt}
\begin{split}
F^h&(x,u(x),u(x)-u(\cdot)) = \\ &\max\left\{-{\det}^h(D^2u(x)) + \kappa(x) \left(1 + \abs{\Dt^h u(x)}^2\right)^{(n+2)/2}, -\lambda_1^h(D^2u(x))\right\},
\end{split}\eq
which involves only the monotone operations of addition, multiplication, and computing the maximum.  As long as the boundary of the domain is sufficiently well-resolved (with spatial resolution on the order of $h\,d\theta$), the resulting scheme is consistent and monotone~\cite[Theorem~16]{FroeseMeshfreeEigs}.

\subsection{Discretisation at boundary}\label{sec:discBoundary}

We also need to define the approximation scheme at the boundary.  At first glance, this appears very challenging since the correct boundary values are not known \emph{a priori} and the weak formulation is influenced by the behaviour of higher-derivatives at the boundary.  Remarkably, though, it is sufficient to enforce the Dirichlet boundary condition in a strong sense:
\bq\label{eq:approxBC}
F^h(x,u(x),u(x)-u(\cdot)) = u(x)-g(x), \quad x \in \G^h\cap\partial\Omega.
\eq

If the weak solution of the PDE is discontinuous at the boundary, this will necessarily introduce a boundary layer into the solution.  However, our interior convergence result (and $L^p$ convergence) will still hold.  To demonstrate this, it is necessary to show that the strong form of the boundary condition actually satisfies the necessary consistency condition.

\begin{lem}[Consistency with the weak Dirichlet problem]\label{lem:consistency}
The approximation scheme $F^h$ defined by~\eqref{eq:approxInt}-\eqref{eq:approxBC} is consistent, monotone, and Lipschitz.
\end{lem}

\begin{proof}
From~\cite[Theorem~16]{FroeseMeshfreeEigs}, the scheme is consistent and monotone in $\Omega$.  It is also trivially monotone on $\partial\Omega$.  By construction, the scheme is Lipschitz as it involves only addition, multiplication, and computing the maximum of operators.  It remains to verify consistency at the boundary.

Consider $x\in\partial\Omega$, smooth $\phi$, and sequences $h_n\to0$, $y_n\in\G^h\cap\partial\Omega$, $z_n\in\G^h\cap\Omega$, $w_n\in\bar{\Omega}$ such that $h_n\to0$ and $y_n, z_n, w_n \to x$.  For sequences that approach $x$ along the boundary we have
\begin{align*}
\lim_{n\to\infty,\xi\to0} &F^{h_n}(y_n,\phi(w_n)+\xi,\phi(w_n)-\phi(\cdot)) = \lim_{n\to\infty,\xi\to0}(\phi(w_n)+\xi-g(y_n))\\
  &= \phi(x)-g(x).
\end{align*}
For sequences that approach $x$ from the interior, we can use the consistency and Lipschitz continuity of the interior approximation to calculate
\begin{align*}
\lim_{n\to\infty,\xi\to0}& F^{h_n}(z_n,\phi(w_n)+\xi,\phi(w_n)-\phi(\cdot)) \\
  &=\max\{-{\det}^+(D^2\phi(x)) + \kappa(x) (1 + \abs{\nabla \phi(x)}^2)^{(n+2)/2}, -\lambda_1(D^2\phi(x))\}\\
	&=F(x,\phi(x),\nabla\phi(x),D^2\phi(x)).
\end{align*}

Combining these results yields
\begin{align*}
\limsup_{h\to0,w\to x, z\in\G^h\to x,\xi\to0} & F^h(z,\phi(w)+\xi,\phi(w)-\phi(\cdot)) \\ 
&=\max\left\{\phi(x)-g(x),F(x,\phi(x),\nabla\phi(x),D^2\phi(x))\right\}\\
&= F^*(x,\phi(x),\nabla\phi(x),D^2\phi(x)).
\end{align*}
We can similarly verify the condition on the limit inferior of the scheme, which establishes consistency in the sense of Definition~\ref{def:consistency}.
\end{proof}

\subsection{Convergence}\label{sec:convergenceGauss}

We now establish that the generalised finite difference method~\eqref{eq:approxInt}-\eqref{eq:approxBC} correctly approximates generalised solutions of the prescribed Gaussian curvature equation.  We begin by demonstrating that the scheme is well-posed, which requires us to construct strict sub- and super-solutions of the PDE.

\begin{lem}[Strict classical super-solution]\label{lem:superExists}
Under the conditions of Hypothesis~\ref{hyp}, the equation of prescribed Gaussian curvature has a strict classical super-solution.
\end{lem}

\begin{proof}
We propose the function
\[ w(x) = -\frac{\abs{x}^2}{2} + M_1, \quad M_1 > \max\limits_{x\in\partial\Omega}\left\{\abs{g(x)} + \frac{\abs{x}^2}{2}\right\}>0. \]
At interior points, we substitute $w$ into the PDE~\eqref{eq:MAconvex} to obtain
\[ F(x,w(x),\nabla w(x),D^2w(x)) \geq -\lambda_1(D^2w(x)) = 1 > 0. \]
At boundary points we have
\[ F^*(x,w(x),\nabla w(x),D^2w(x)) \geq w(x)-g(x) > \max\limits_{y\in\partial\Omega}\left\{\abs{g(y)}\right\} - g(x) > 0,\]
which establishes $w$ as a strict super-solution.
\end{proof}

\begin{lem}[Strict classical sub-solution]\label{lem:subExists}
Under the conditions of Hypothesis~\ref{hyp}, the equation of prescribed Gaussian curvature has a strict classical sub-solution.
\end{lem}

\begin{proof}
Our approach is to increase the curvature and smooth the data in order to produce a smooth, strict sub-solution of the PDE.  

Since the given curvature $\kappa$ satisfies the strict compatibility condition of Hypothesis~\ref{hyp}, we can choose $\gamma>0$ such that
\[ \int_\Omega(\kappa(x)+\gamma)\,dx < \int_{\R^n}(1+\abs{p}^2)^{-(n+2)/2}\,dp. \]

For any $\epsilon>0$, we can define an enlarged, uniformly continuous domain that includes an $\epsilon$-neighbourhood of the original uniformly continuous domain:
\[ \Omega_\epsilon = \{x\in\R^n\mid\text{dist}(x,\bar{\Omega})<\epsilon\}. \]
We can also extend the curvature $\kappa$ into this domain via
\[ \kappa_\epsilon(x) = \left(\kappa\left(\text{Proj}_{\bar{\Omega}}x\right)+\gamma\right)\min\left\{\frac{1}{\epsilon}\text{dist}(x,\partial\Omega_\epsilon),1\right\}. \]
Notice that $\kappa_\epsilon>0$ in $\Omega_\epsilon$, $\kappa_\epsilon=0$ on $\partial{\Omega}_\epsilon$, and $\kappa_\epsilon = \kappa+\gamma$ in $\Omega$.  We can further mollify this to produce a $C^2$ curvature function~$\tilde{\kappa}_\epsilon$ satisfying
\[ \abs{\tilde{\kappa}_\epsilon-\kappa_\epsilon} < \epsilon. \]

For sufficiently small $\epsilon>0$, we have both of the following conditions:
\[ \tilde{\kappa}_\epsilon(x) >  \kappa(x), \quad x \in \Omega \]
\[ \int_{\Omega_\epsilon}\tilde{\kappa}_\epsilon(x)\,dx < \int_{\R^n}(1+\abs{p}^2)^{-(n+2)/2}\,dp. \]

We also choose a constant 
\[ g_0 < \min\limits_{x\in\partial\Omega} g(x). \]

This smoothed data ensures that the following PDE has a smooth convex solution~$v$~\cite[Corollary~17.25]{GilTrudBook}.
\[
\begin{cases}
\det(D^2v(x)) = \tilde{\kappa}_\epsilon(x)(1+\abs{\nabla v(x)}^2)^{(n+2)/2}, & x\in\Omega_\epsilon\\
v(x) = g_0, & x \in \partial\Omega_\epsilon.
\end{cases}
\]

We note that since $v$ is convex, it attains its maximum on the boundary and thus $v \leq g_0$ in $\Omega_\epsilon$.

Now we verify that $v$ is a sub-solution of~\eqref{eq:MAconvex}.  Consider first interior points $x\in\Omega$.  Since $v$ is convex and its curvature satisfies $\tilde{\kappa}_\epsilon(x) > \kappa(x) \geq 0$, we have
\[ -\lambda_1(D^2v(x)) < 0. \]
The other term in the maximum is
\[ -{\det}^+(D^2v(x))+\kappa(x)R(\nabla v(x)) < -{\det}(D^2v(x)) + \tilde{\kappa}_\epsilon(x)R(\nabla v(x)) = 0. \]
We conclude that
\[ \max\left\{-{\det}^+(D^2v(x))+\kappa(x)R(\nabla v(x)),-\lambda_1(D^2v(x))\right\} < 0. \qedhere\]
\end{proof}

Now we combine these two lemmas with our earlier equivalence results (Theorems~\ref{thm:exist} and~\ref{thm:uniqueness}), the interior comparison principle (Theorem~\ref{thm:comparison}), well-posedness (Lemma~\ref{lem:existDiscrete}) and convergence (Theorem~\ref{thm:convergence}) criteria for schemes, and the consistency, monotonicity, and Lipschitz continuity of the scheme (Lemma~\ref{lem:consistency}).  The main result of this article is the following.

\begin{thm}[Convergence]\label{thm:filtered}
Consider the Gaussian curvature equation~\eqref{eq:MA} with data satisfying Hypothesis~\ref{hyp} and let $u(x)$ be the unique generalised solution of the equation.  For sufficiently small $h>0$, the approximation scheme~\eqref{eq:approxInt}-\eqref{eq:approxBC} has a solution~$U^h$, which defines a nearest neighbour extension $u^h(x)$ via~\eqref{eq:extension}.  Consider any $x\in\Omega$.  Then as $h\to0$, $u^h(x)\to u(x)$. 
\end{thm} 


\subsection{Examples}\label{sec:compute}
We conclude by presenting several computational examples that demonstrate that this monotone scheme does indeed correctly compute surfaces of prescribed Gaussian curvature.  The examples we present are two-dimensional, non-classical viscosity solutions, but the convergence results presented in this article are valid in any dimension.

Each of the following examples is posed on a domain $\Omega\subset\R^2$ that consists of the half of the unit disc where $x>0$.  The equations are discretised on a uniform Cartesian grid with spacing $h$ inside the domain, augmented by approximately $4/h^{3/2}$ points on the boundary  $\partial\Omega$.  The resulting point cloud is pictured in Figure~\ref{fig:mesh}.  The discrete equations were solved in Matlab using an explicit iterative scheme as in~\cite{ObermanEigenvalues}.

Our purpose here is to validate the convergence of monotone schemes rather than to produce an optimal method.  However, we expect that the techniques of~\cite{FroeseMeshfreeEigs,FOFiltered} can be adapted to produce significantly more accurate almost-monotone schemes on meshes that are adapted to resolve boundary layers.


\subsubsection{Lipschitz ($C^{0,1}$)}\label{sec:lipschitz}

The first example we consider is given by
\[ u(x,y) = \abs{-x\sin(\pi/10)+y\cos(\pi/10)}, \quad \kappa(x,y) = 0. \]
This example, which is Lipschitz continuous but not differentiable, is pictured in Figure~\ref{fig:deg}.  We note that the line of non-differentiability does not align with any grid direction.  Because the curvature vanishes, the ellipticity of the \MA equation~\eqref{eq:MA} is everywhere degenerate.  Nevertheless, the monotone scheme converges to the correct solution as is seen in Table~\ref{table:error}.  

We remark that for such a singular example, the error is not expected to decrease monotonically to zero as the grid is refined.  This is because the directions that we resolve are distributed uniformly on the unit ball.  While the number of directions is increased as the grid is refined, it is possible that a direction that is present on a coarse grid (but not on the next refinement) may happen to align well with the singularity, which can lead to an unusually low error.  However, we do expect the error to converge to zero as the grid is refined.  This is observed in Table~\ref{table:error}.

\subsubsection{Unbounded gradient ($C^0$)}\label{sec:blowup}

For our second example, we consider the constant-curvature surface of the unit ball
\[ u(x,y) = -\sqrt{1-x^2-y^2}, \quad \kappa(x,y) = 1. \]
This solution is continuous, but the gradient blows up along a portion of the boundary; see Figure~\ref{fig:ballCts}.  Despite the low regularity of this solution, the monotone scheme correctly computes this weak solution.  See Table~\ref{table:error}.

\subsubsection{Dirichlet data not attained}\label{sec:discts}

Finally, we consider a modification of the previous example that does not satisfy the Dirichlet boundary conditions in a classical sense.  We again look for a surface of constant unit curvature with the data
\[ g(x,y) = -\sqrt{1-x^2-y^2} + \frac{1}{4}x, \quad \kappa(x,y) = 1. \]
The exact solution is again the surface of the unit ball,
\[ u(x,y) = -\sqrt{1-x^2-y^2}, \]
which does not agree with $g(x,y)$ on much of the boundary.  
See Figure~\ref{fig:ballDiscts} for the computed solution, which lies strictly above the true solution (Figure~\ref{fig:ballCts}) at the boundary because of the strong implementation of the Dirichlet boundary conditions.

In this case, it is not possible to obtain convergence in $L^\infty$ since the computed solution must contain an error of 0.25 at the boundary.  The resulting boundary layer is evident in the plot of error in Figure~\ref{fig:disctsErr}. However, as predicted by Corollary~\ref{cor:Lp}, we do observe convergence in $L^1$ despite the highly non-classical nature of this example.  See Table~\ref{table:error}.

\begin{figure}
\centering
	{\includegraphics[width=0.4\textwidth]{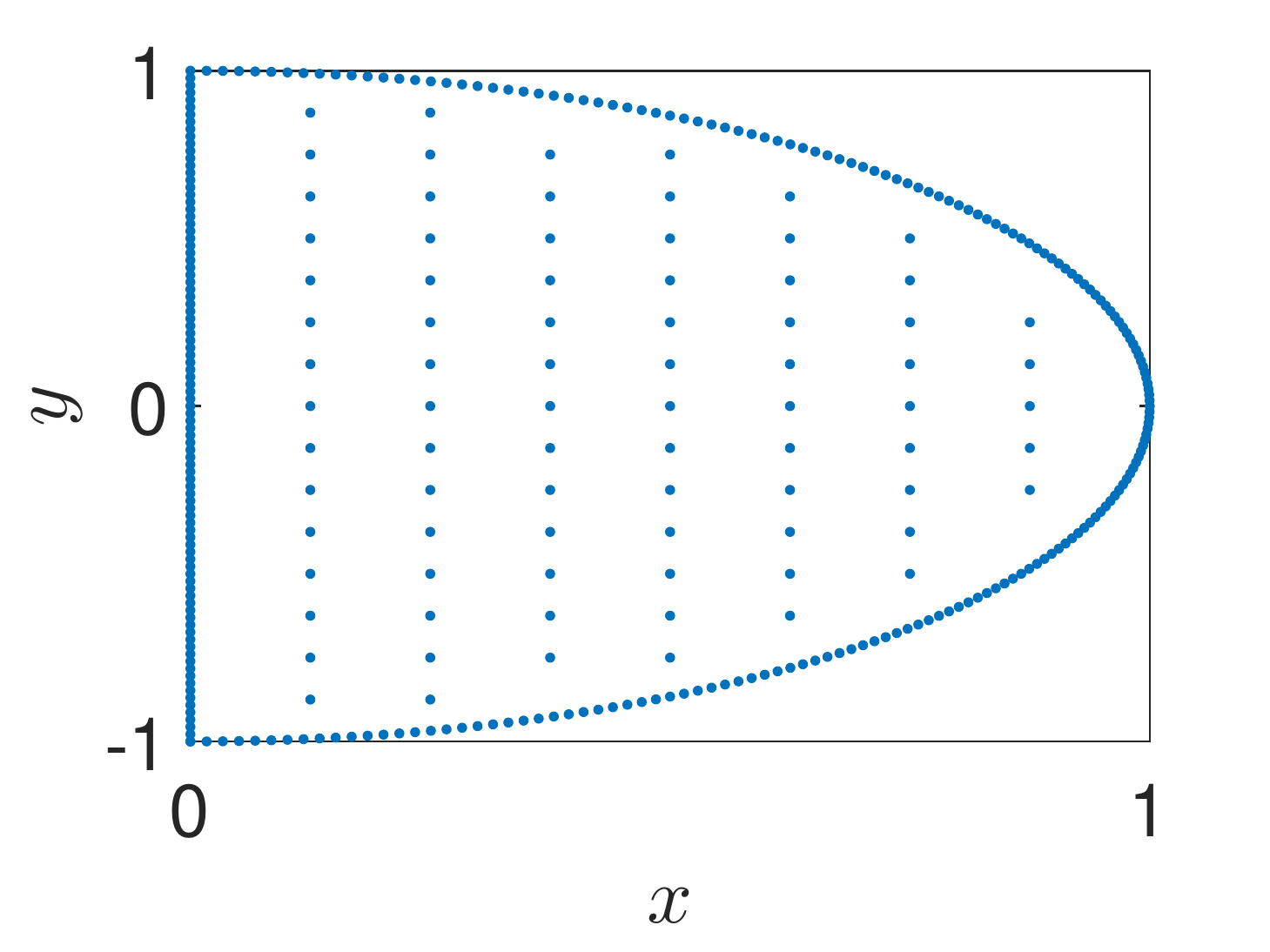}}
\caption{Computational point cloud with $h=2^{-3}$.}
\label{fig:mesh}
\end{figure}

\begin{figure}[tbhp]
  \centering
  \subfigure[]{\label{fig:deg}\includegraphics[width=0.45\textwidth]{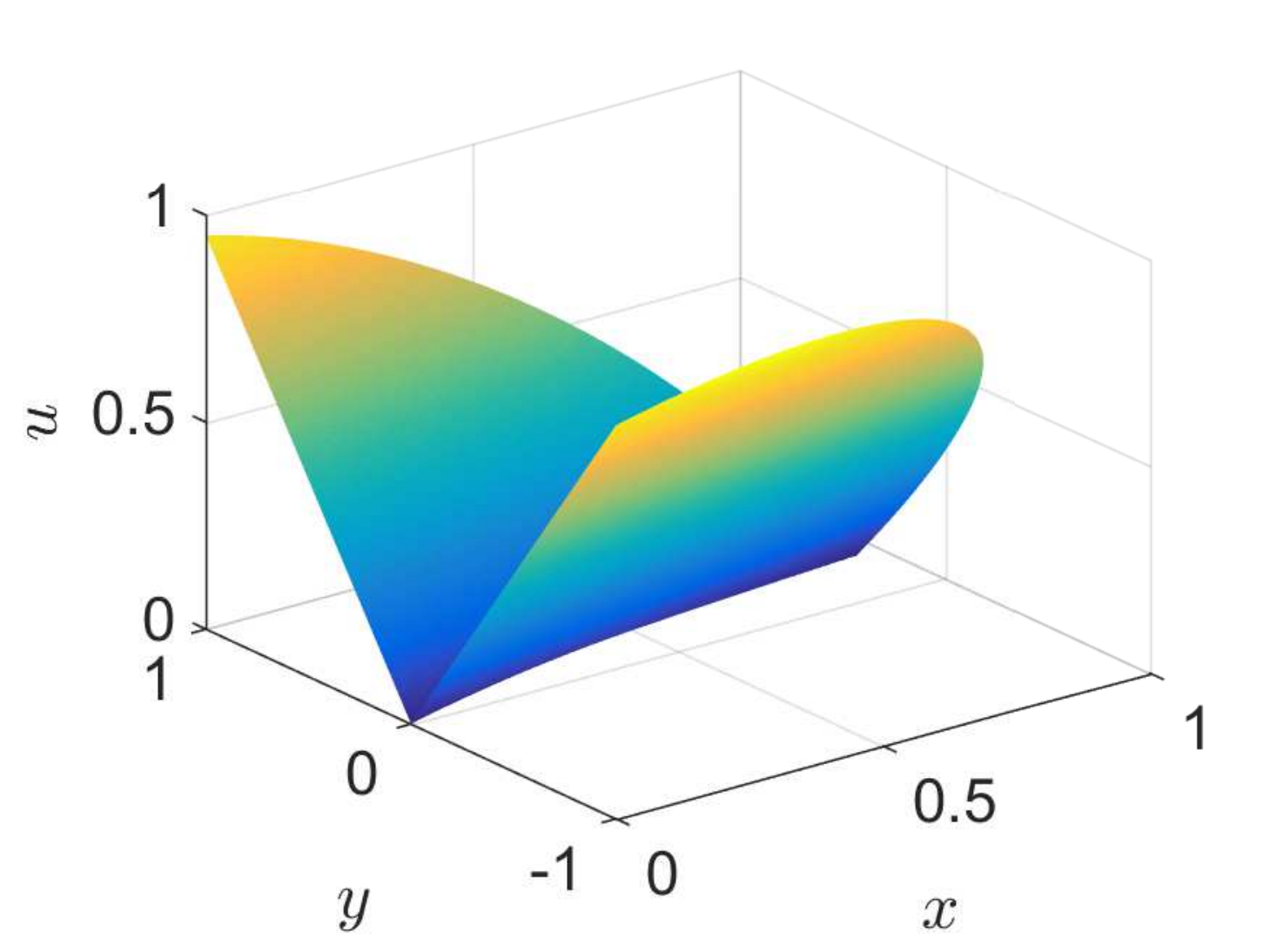}}
  \subfigure[]{\label{fig:ballCts}\includegraphics[width=0.45\textwidth]{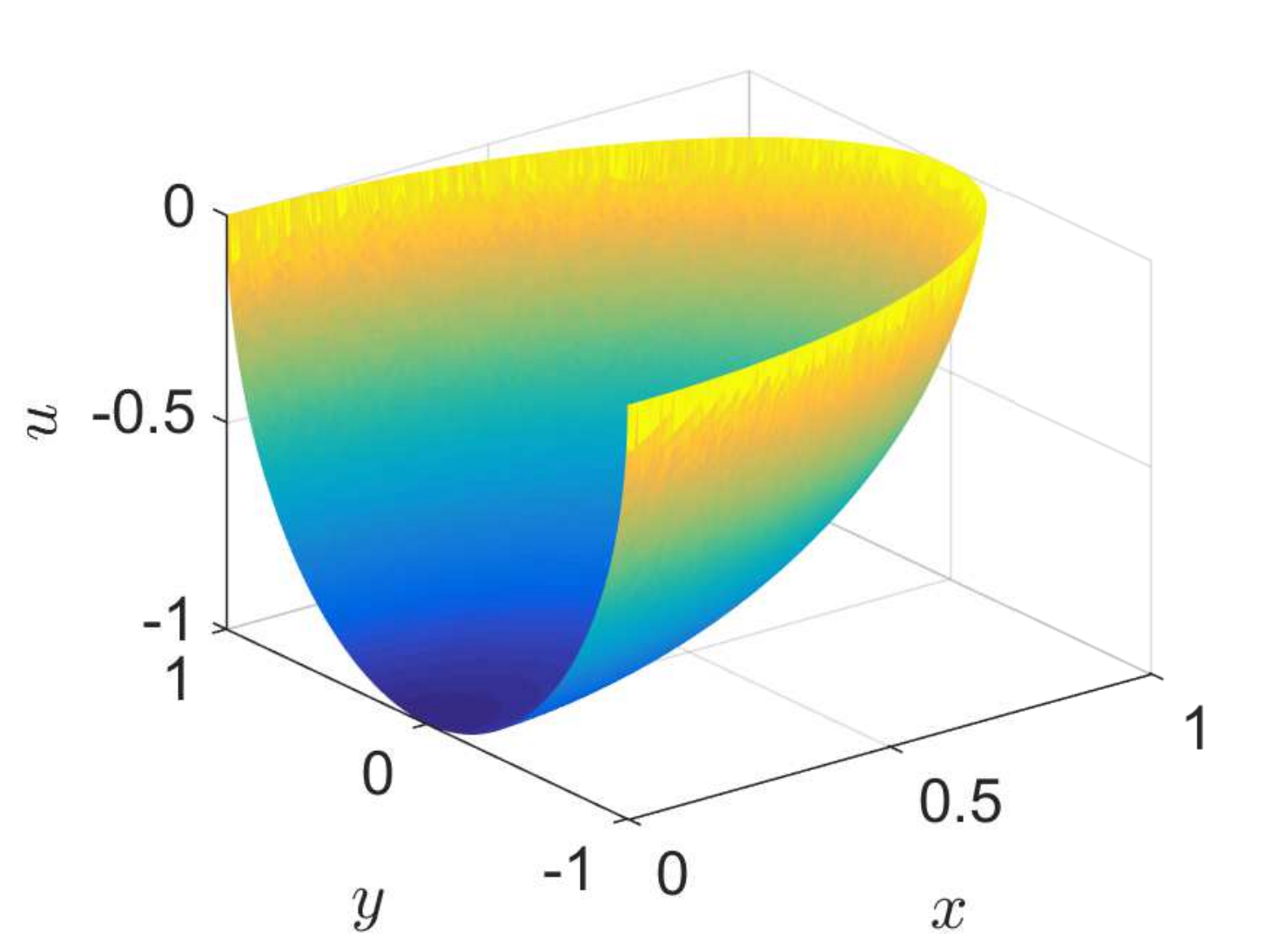}}
	\subfigure[]{\label{fig:ballDiscts}\includegraphics[width=0.45\textwidth]{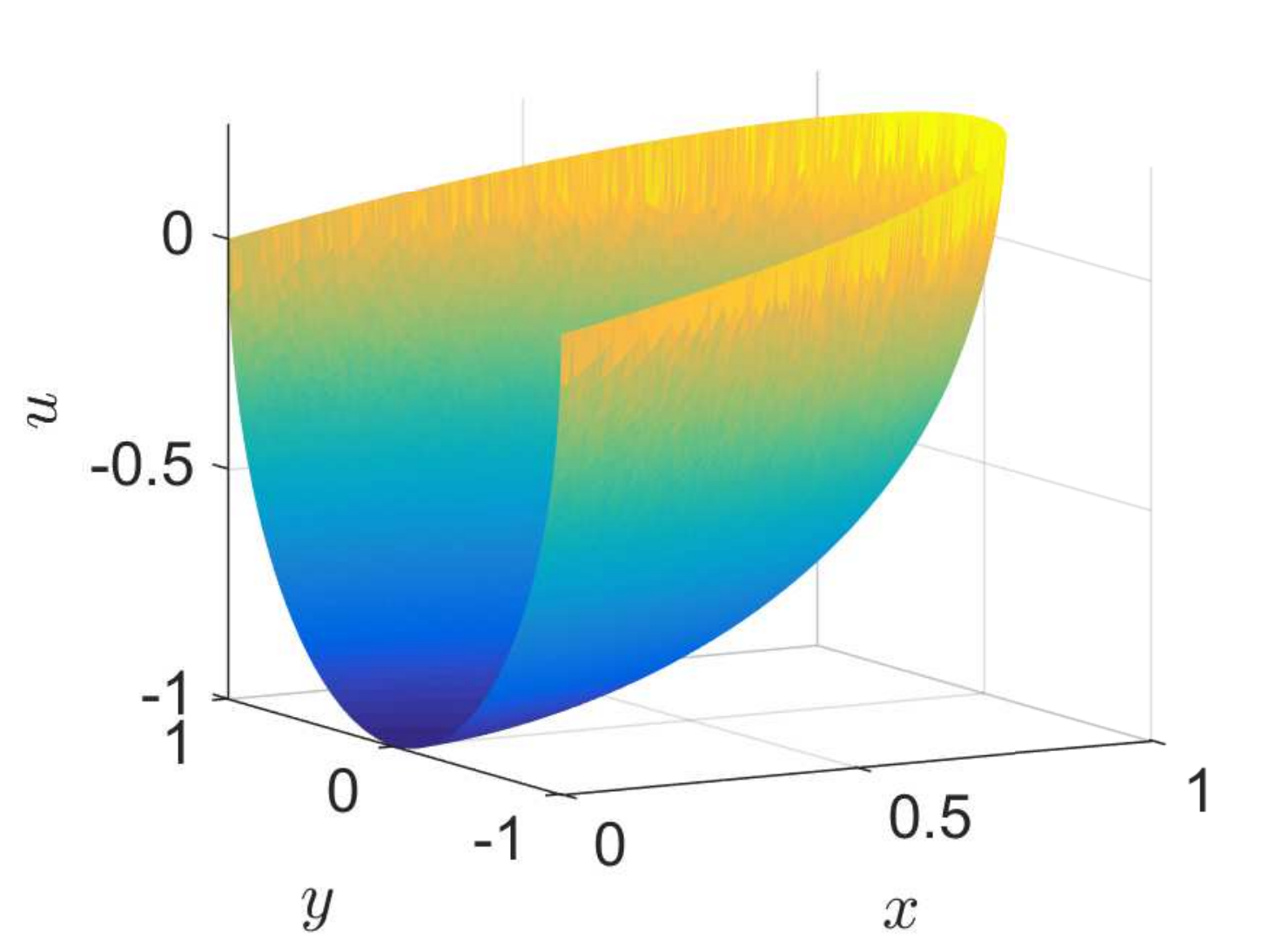}}
	\subfigure[]{\label{fig:disctsErr}\includegraphics[width=0.43\textwidth]{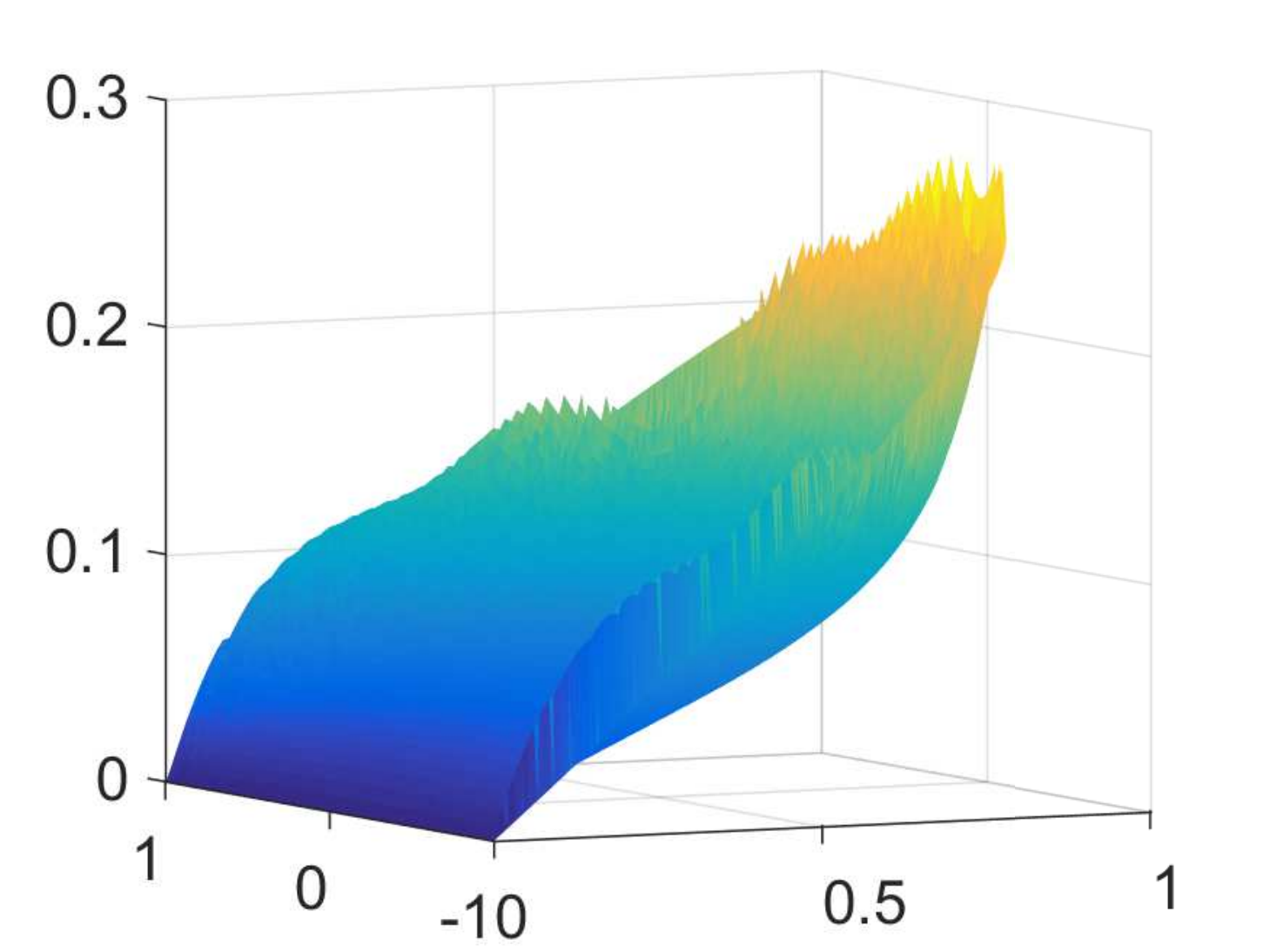}}
  \caption{Computed approximations ($h=2^{-7}$) to solutions that \subref{fig:deg}~are Lipschitz continuous (\ref{sec:lipschitz}), \subref{fig:ballCts}~have an unbounded gradient (\ref{sec:blowup}), and \subref{fig:ballDiscts}~do not achieve the Dirichlet data (\ref{sec:discts}). \subref{fig:disctsErr}~Error in discontinuous solution.}
  \label{fig:solutions}
\end{figure}

\begin{table}
\caption{Error in computed solutions.}
\label{table:error}
\centering
\begin{tabular}{c|c|c|cc}
& $C^{0,1}$ & $C^0$ & \multicolumn{2}{c}{Non-continuous}\\
$h$ & $\|u-u^h\|_\infty$ & $\|u-u^h\|_\infty$ & $\|u-u^h\|_\infty$ & $\|u-u^h\|_1$\\ \hline
$2^{-3}$ & $9.45\times10^{-2}$ & $1.94\times10^{-1}$ & $3.55\times10^{-1}$ & $2.12\times10^{-1}$  \\
$2^{-4}$ & $9.27\times10^{-2}$ & $1.61\times10^{-1}$ & $3.33\times10^{-1}$ & $1.83\times10^{-1}$ \\
$2^{-5}$ & $6.48\times10^{-2}$ & $1.28\times10^{-1}$ & $3.05\times10^{-1}$ & $1.60\times10^{-1}$ \\
$2^{-6}$ & $6.41\times10^{-2}$ & $1.09\times10^{-1}$ & $2.90\times10^{-1}$ & $1.33\times10^{-1}$ \\
$2^{-7}$ & $3.18\times10^{-2}$ & $8.80\times10^{-2}$ & $2.74\times10^{-1}$ & $9.53\times10^{-2}$ \\
\end{tabular}
\end{table}

\section{Conclusions}\label{sec:conclude}
In this article, we developed  a proof that surfaces of prescribed Gaussian curvature can be constructed through the use of monotone approximations of a Monge-Amp\`ere type equation.  

Typical convergence proofs for the approximation of weak (viscosity) solutions of nonlinear degenerate elliptic equations require on a comparison principle that ensures that sub-solutions lie below super-solutions.  However, this is demonstrably false for our equation equipped with Dirichlet boundary conditions, which must be interpreted in a weak sense and which allow for solutions that are discontinuous at the boundary.  

By relying on a geometric interpretation of the \MA equation, we proved that a comparison principle does hold in the interior of the domain.  This result relied on the fact that the sub-gradients of viscosity solutions can be ordered even when the boundary conditions are not satisfied in the classical sense.  Using this comparison result, we modified the traditional Barles-Souganidis framework to prove that consistent, monotone schemes are well-posed and converge to the non-continuous viscosity solution, though possibly with a boundary layer.

To validate these results, we implemented a monotone scheme for the prescribed Gaussian curvature equation in two-dimensions.  Convergence to the viscosity solution was observed for challenging examples including a Lipschitz continuous solution, a solution with unbounded gradient, and a solution that did not satisfy the boundary conditions in a classical sense.

The monotone scheme presented here is low-accuracy, but formally higher-order filtered schemes can be constructed as in~\cite{FOFiltered}.  The schemes can also be adapted to non-uniform meshes as in~\cite{FroeseMeshfreeEigs}.  A natural extension would be to introduce filtered schemes that are one-sided and highly resolved near the boundary in an attempt to reduce the effects of the boundary layer.  Because solutions have very low regularity, fast solvers such as Newton's method are not effective, and the development of more appropriate solution methods is another possible direction for future work.

\bibliographystyle{plain}
\bibliography{GaussCurvature}

\begin{thebibliography}{10}

\bibitem{Bakelman}
I.~J. Bakelman.
\newblock Generalized elliptic solutions of the {D}irichlet problem for
  n-dimensional {M}onge-{A}mp\`ere equations.
\newblock In {\em Nonlinear Functional Analysis and its Applications},
  volume~45 of {\em P. Symp. Pure Math.}, pages 73--102. AMS, 1986.

\bibitem{Bakelman_Elliptic}
I.~J. Bakelman.
\newblock {\em Convex analysis and nonlinear geometric elliptic equations}.
\newblock Springer Science \& Business Media, 2012.

\bibitem{BardiDolcetta}
M.~Bardi and I.~Capuzzo-Dolcetta.
\newblock {\em Optimal control and viscosity solutions of
  Hamilton-Jacobi-Bellman equations}.
\newblock Springer Science \& Business Media, 2008.

\bibitem{BardiMannucci}
M.~Bardi and P.~Mannucci.
\newblock Comparison principles and {D}irichlet problem for fully nonlinear
  degenerate equations of {M}onge-{A}mp{\`e}re type.
\newblock {\em Forum Math.}, 25(6):1291--1330, 2013.

\bibitem{BSnum}
G.~Barles and P.~E. Souganidis.
\newblock Convergence of approximation schemes for fully nonlinear second order
  equations.
\newblock {\em Asymptotic Anal.}, 4(3):271--283, 1991.

\bibitem{benamou2014monotone}
J.-D. Benamou, F.~Collino, and J.-M. Mirebeau.
\newblock Monotone and consistent discretization of the {M}onge-{A}mpere
  operator.
\newblock {\em Mathematics of computation}, 85(302):2743--2775, 2016.

\bibitem{Blocki}
Z.~B{\l}ocki.
\newblock On the {D}arboux equation.
\newblock {\em Zeszyty Naukowe Uniwersytetu Jagiello{\'n}skiego. Universitatis
  Iagellonicae Acta Mathematica}, 1255:87--90, 2001.

\bibitem{BrennerNeilanMA2D}
S.~C. Brenner, T.~Gudi, M.~Neilan, and L.-Y. Sung.
\newblock ${C}^0$ penalty methods for the fully nonlinear {M}onge-{A}mp\`ere
  equation.
\newblock {\em Math. Comp.}, 80(276):1979--1995, 2011.

\bibitem{CafNirSpruck}
L.~Caffarelli, L.~Nirenberg, and J.~Spruck.
\newblock The {D}irichlet problem for nonlinear second-order elliptic equations
  i. {M}onge-{A}mp{\'e}re equation.
\newblock {\em Comm. Pure Appl. Math.}, 37(3):369--402, 1984.

\bibitem{caffarelli_eigs}
L.~Caffarelli, L.~Nirenberg, and J.~Spruck.
\newblock The {D}irichlet problem for nonlinear second order elliptic
  equations, {III}: Functions of the eigenvalues of the {H}essian.
\newblock {\em Acta Mathematica}, 155(1):261--301, 1985.

\bibitem{CIL}
M.~G. Crandall, H.~Ishii, and P.-L. Lions.
\newblock User's guide to viscosity solutions of second order partial
  differential equations.
\newblock {\em Bull. Amer. Math. Soc. (N.S.)}, 27(1):1--67, 1992.

\bibitem{DGnum2006}
E.~J. Dean and R.~Glowinski.
\newblock Numerical methods for fully nonlinear elliptic equations of the
  {M}onge-{A}mp\`ere type.
\newblock {\em Comput. Methods Appl. Mech. Engrg.}, 195(13-16):1344--1386,
  2006.

\bibitem{ElseyEsedoglu}
M.~Elsey and S.~Esedo{\=g}lu.
\newblock Analogue of the total variation denoising model in the context of
  geometry processing.
\newblock {\em Multiscale Model. Simul.}, 7(4):1549--1573, 2009.

\bibitem{FengNeilan}
X.~Feng and M.~Neilan.
\newblock Vanishing moment method and moment solutions for fully nonlinear
  second order partial differential equations.
\newblock {\em J. Sci. Comput.}, 38(1):74--98, 2009.

\bibitem{DelzannoGrid}
J.~M. Finn, G.~L. Delzanno, and L.~Chac{\'o}n.
\newblock Grid generation and adaptation by {M}onge-{K}antorovich optimization
  in two and three dimensions.
\newblock In {\em Proc. 17th Int. Meshing Roundtable}, pages 551--568, 2008.

\bibitem{FroeseTransport}
B.~D. Froese.
\newblock A numerical method for the elliptic {M}onge-{A}mp\`ere equation with
  transport boundary conditions.
\newblock {\em SIAM J. Sci. Comput.}, 34(3):A1432--A1459, 2012.

\bibitem{FroeseMeshfreeEigs}
B.~D. Froese.
\newblock Meshfree finite difference approximations for functions of the
  eigenvalues of the {Hessian}.
\newblock Submitted, \url{http://arxiv.org/pdf/1512.06287v1.pdf}, 2015.

\bibitem{FO_MATheory}
B.~D. Froese and A.~M. Oberman.
\newblock Convergent finite difference solvers for viscosity solutions of the
  elliptic {M}onge-{A}mp\`ere equation in dimensions two and higher.
\newblock {\em SIAM J. Numer. Anal.}, 49(4):1692--1714, 2011.

\bibitem{FOFiltered}
B.~D. Froese and A.~M. Oberman.
\newblock Convergent filtered schemes for the {M}onge- {A}mp\`ere partial
  differential equation.
\newblock {\em SIAM J. Numer. Anal.}, 51(1):423--444, 2013.

\bibitem{GilTrudBook}
D.~Gilbarg and N.~S. Trudinger.
\newblock {\em Elliptic partial differential equations of second order}, volume
  224 of {\em Grundlehren Math. Wiss.}
\newblock Springer-Verlag, 2nd edition, 1983.

\bibitem{Gutierrez}
C.~E. Guti{\'e}rrez.
\newblock {\em The {M}onge--{A}mp{\`e}re Equation}, volume~44 of {\em Progr.
  Nonlinear Differential Equations Appl.}
\newblock Springer Science \& Business Media, 2001.

\bibitem{Han_isometric}
Q.~Han and J.-X. Hong.
\newblock {\em Isometric embedding of {R}iemannian manifolds in {E}uclidean
  spaces}, volume 130.
\newblock American Mathematical Society Providence, 2006.

\bibitem{IshiiLions}
H.~Ishii and P.-L. Lions.
\newblock Viscosity solutions of fully nonlinear second-order elliptic partial
  differential equations.
\newblock {\em J. Diff. Eq.}, 83(1):26--78, 1990.

\bibitem{KruskalCE}
J.~B. Kruskal.
\newblock Two convex counterexamples: A discontinuous envelope function and a
  nondifferentiable nearest-point mapping.
\newblock {\em Proc. Amer. Math. Soc.}, pages 697--703, 1969.

\bibitem{Lions_Remarks}
P.-L. Lions.
\newblock Two remarks on {M}onge-{A}mpere equations.
\newblock {\em Ann. Mat. Pura Appl.}, 142(1):263--275, 1985.

\bibitem{LoeperMA}
G.~Loeper and F.~Rapetti.
\newblock Numerical solution of the {M}onge-{A}mp\'{e}re equation by a
  {N}ewton's algorithm.
\newblock {\em C. R. Math. Acad. Sci. Paris}, 340(4):319--324, 2005.

\bibitem{mirebeau2015MA}
J.-M. Mirebeau.
\newblock Discretization of the 3d {M}onge-{A}mpere operator, between wide
  stencils and power diagrams.
\newblock {\em ESAIM: Mathematical Modelling and Numerical Analysis},
  49(5):1511--1523, 2015.

\bibitem{ObermanCE}
A.~Oberman.
\newblock The convex envelope is the solution of a nonlinear obstacle problem.
\newblock {\em Proc. Amer. Math. Soc.}, 135(6):1689--1694, 2007.

\bibitem{ObermanSINUM}
A.~M. Oberman.
\newblock Convergent difference schemes for degenerate elliptic and parabolic
  equations: {H}amilton--{J}acobi equations and free boundary problems.
\newblock {\em SIAM J. Numer. Anal.}, 44(2):879--895, 2006.

\bibitem{ObermanEigenvalues}
A.~M. Oberman.
\newblock Wide stencil finite difference schemes for the elliptic
  {M}onge-{A}mp\`ere equation and functions of the eigenvalues of the
  {H}essian.
\newblock {\em Discrete Contin. Dyn. Syst. Ser. B}, 10(1):221--238, 2008.

\bibitem{oliker2007embedding}
V.~Oliker.
\newblock Embedding ${S}^n$ into ${R}^{n+ 1}$ with given integral {G}auss
  curvature and optimal mass transport on ${S}^n$.
\newblock {\em Advances in Mathematics}, 213(2):600--620, 2007.

\bibitem{olikerprussner88}
V.~I. Oliker and L.~D. Prussner.
\newblock On the numerical solution of the equation $(\partial^2z/\partial
  x^2)(\partial^2z/\partial y^2)-(\partial^2z/\partial x\partial y)^2=f$ and
  its discretizations, {I}.
\newblock {\em Numer. Math.}, 54(3):271--293, 1988.

\bibitem{OsherSethian}
S.~Osher and J.~A. Sethian.
\newblock Fronts propagating with curvature-dependent speed: algorithms based
  on {H}amilton-{J}acobi formulations.
\newblock {\em J. Comput. Phys.}, 79(1):12--49, 1988.

\bibitem{Sapiro}
G.~Sapiro.
\newblock {\em Geometric partial differential equations and image analysis}.
\newblock Cambridge University Press, 2006.

\bibitem{Saumier}
L.-P. Saumier, M.~Agueh, and B.~Khouider.
\newblock An efficient numerical algorithm for the {L}2 optimal transport
  problem with periodic densities.
\newblock {\em IMA J. Appl. Math.}, 80(1):135--157, 2015.

\bibitem{SulmanWilliamsRussell}
M.~Sulman, J.~F. Williams, and R.~D. Russell.
\newblock Optimal mass transport for higher dimensional adaptive grid
  generation.
\newblock {\em J. Comput. Phys.}, 230(9):3302--3330, 2011.

\bibitem{TrudUrbas_Gauss}
N.~S. Trudinger and J.~I.~E. Urbas.
\newblock The {D}irichlet problem for the equation of prescribed {G}auss
  curvature.
\newblock {\em Bull. Aust. Math. Soc.}, 28(02):217--231, 1983.

\bibitem{TrudingerWang_MAReview}
N.~S. Trudinger and X.-J. Wang.
\newblock The {M}onge-{A}mp\`ere equation and its geometric applications.
\newblock In {\em Handbook of Geometric Analysis}, volume~7 of {\em Adv. Lect.
  Math.}, pages 467--524. Int. Press, 2008.

\bibitem{Urbas_MA}
J.~I.~E. Urbas.
\newblock The generalized {D}irichlet problem for equations of {M}onge-{A}mpere
  type.
\newblock {\em Annales de l'IHP Analyse non lin{\'e}aire}, 3(3):209--228, 1986.

\bibitem{Villani}
C.~Villani.
\newblock {\em Topics in optimal transportation}, volume~58 of {\em Graduate
  Studies in Mathematics}.
\newblock AMS, Providence, RI, 2003.

\bibitem{Zhao}
H.~Zhao.
\newblock A fast sweeping method for {E}ikonal equations.
\newblock {\em Math. Comp.}, 74(250):603--627, 2005.

\end{thebibliography}

{\appendix

\section{Equivalence of weak solutions}\label{app:equivalence}

In this appendix, we modify the results of~\cite{Gutierrez} in order to prove Theorem~\ref{thm:equivalence1}, which asserts that generalised and viscosity solutions of~\eqref{eq:curvature} are equivalent on open sets.

\begin{lem}[Generalised solutions are viscosity solutions]\label{lem:aleksToVisc}
Consider a domain $\Omega$ and curvature~$\kappa$ satisfying the conditions of Hypothesis~\ref{hyp}.  Let~$u$ be a convex generalised solution of~\eqref{eq:curvature}.  Then $u$ is a viscosity solution of~\eqref{eq:curvature}.
\end{lem}

\begin{proof}
We demonstrate here that $u$ is a viscosity sub-solution.  The proof that it is a super-solution, and therefore a viscosity solution, is similar.

Choose any $x_0\in\Omega$ and  $\phi\in C^2$ such that $u-\phi$ has a strict local maximum at $x_0$.  As in the proof of Lemma~\ref{lem:sub12}, this implies that there is some $\delta>0$ such that $u(x) < \phi(x)$ and $D^2\phi(x) \geq 0$ whenever $0 < \abs{x-x_0} \leq \delta$.  We define the positive constant
\[ m = \min\limits_{\delta/2 \leq \abs{x-x_0} \leq \delta}\{\phi(x)-u(x)\} > 0. \]

Now choose any $0 < \epsilon < m$ and define the set
\[ S_\epsilon = \{x \mid \abs{x-x_0} \leq \delta, \, u(x) + \epsilon > \phi(x)\}. \]
We notice that whenever $\delta/2 \leq \abs{x-x_0} \leq \delta$, we must have $\phi(x)-u(x) \geq m > \epsilon$ and $x\notin S_\epsilon$.  We conclude that $S_\epsilon$ is contained in the smaller ball
\[ S_\epsilon \subset B(x_0,\delta/2). \]

Now for any $z\in\partial S_\epsilon$, we can construct sequences $x_n\in S_\epsilon$, $y_n \notin S_\epsilon$ such that both converge to the boundary point $z$: $x_n, y_n \to z$.  By the continuity of $u$ and $\phi$, we must have
\[ u(x) + \epsilon = \phi(x), \quad x\in\partial S_\epsilon. \]

Since $u+\epsilon > \phi$ in $S_\epsilon$ with equality on the boundary $\partial S_\epsilon$, we can apply~\cite[Lemma~1.4.1]{Gutierrez} to obtain
\[ \partial u(S_\epsilon) = \partial(u+\epsilon)(S_\epsilon) \subset \partial\phi(S_\epsilon). \]
Combined with the fact that $u$ is a generalised solution, we obtain
\begin{align*}
\int_{S_\epsilon}\kappa(x)\,dx &= \int_{\partial u(S_\epsilon)}(1+\abs{p}^2)^{-(n+2)/2}\,dp \\
  &\leq \int_{\partial\phi(S_\epsilon)}(1+\abs{p}^2)^{-(n+2)/2}\,dp\\
	&= \int_{S_\epsilon} (1+\abs{\nabla\phi(x)}^2)^{-(n+2)/2}\det(D^2\phi(x))\,dx.
\end{align*}
As this holds for all sufficiently small $\epsilon>0$, we can use the continuity of $\phi$ to obtain
\[ -\det(D^2\phi(x_0)) + \kappa(x_0)(1+\abs{\nabla\phi(x_0)}^2)^{(n+2)/2}\,dx \leq 0, \]
which implies that $u$ is a viscosity sub-solution.
\end{proof}

\begin{lem}[Viscosity solutions are generalised solutions]\label{thm:viscToAleks}
Consider a domain $\Omega$ and curvature~$\kappa$ satisfying the conditions of Hypothesis~\ref{hyp}.  Let~$u$ be a convex viscosity solution of~\eqref{eq:curvature}.  Then $u$ is a generalised solution of~\eqref{eq:curvature}.
\end{lem}

\begin{proof}
Choose any uniformly convex set $E$ such that $\bar{E} \subset \Omega$.  Since $u$ is convex, $u \in C^{0,1}(\bar{E})$. We will verify that 
\[ \int_{\partial u(E)}(1+\abs{p}^2)^{-(n+2)/2}\,dp = \int_E\kappa(x)\,dx. \]

{\bf Super-solution.}
First we construct a super-solution of $u$.  For $\epsilon>0$, define the smaller set
\[ E_\epsilon = \{x\in E \mid \text{dist}(x,\partial E) > \epsilon\}. \]
Next we define the modified curvature function
\[ \kappa_\epsilon(x) = \frac{1}{\epsilon}\kappa\left(x\right)\min\left\{\text{dist}(x,\partial E),\epsilon\right\} \leq \kappa(x). \]
We note that $\kappa = \kappa_\epsilon$ in $E_\epsilon$.

Now we let $v_\epsilon$ be the convex generalised solution of
\bq\label{eq:PDESuper}
\begin{cases}
-\det(D^2v_\epsilon(x)) + \kappa_\epsilon(x)(1+\abs{\nabla v_\epsilon(x)}^2)^{(n+2)/2}, & x\in E\\
v_\epsilon(x) = u(x), & x\in\partial E.
\end{cases}
\eq
By~\cite[Theorem 11.8]{Bakelman_Elliptic}, the solution continuously attains the boundary data: $v_\epsilon\in C^0(\bar{E})$ and $v_\epsilon(x) = u(x)$ on $\partial E$.
Additionally, by Lemma~\ref{lem:aleksToVisc}, $v_\epsilon$ is a viscosity solution (and thus super-solution) of~\eqref{eq:PDESuper}.  

Then we can consider any $x_0\in E$ and $\phi\in C^2$ such that $D^2v(x_0) > 0$ and $v_\epsilon-\phi$ has a local minimum at $x_0$.  Since $v_\epsilon$ is a super-solution of~\eqref{eq:PDESuper}, we must have
\begin{align*}
0 &\leq -\det(D^2\phi(x_0)) + \kappa_\epsilon(x_0)(1+\abs{\nabla\phi(x_0)}^2)^{(n+2)/2}\\ &\leq -\det(D^2\phi(x_0)) + \kappa(x_0)(1+\abs{\nabla\phi(x_0)}^2)^{(n+2)/2} .
\end{align*}
Thus $v_\epsilon$ is also a super-solution of the original equation~\eqref{eq:MA}.

Since $v_\epsilon$ is a super-solution, $u$ a sub-solution, and $v_\epsilon = u$ on $\partial E$, we can use the classical comparison principle (Theorem~\ref{thm:comparisonGauss}, which is a special case of~\cite[Theorem V.2]{IshiiLions}) to deduce that
\[ v_\epsilon(x) \geq u(x), \quad x\in\bar{E}. \]
This allows us to order the subgradients of these functions using~\cite[Lemma~1.4.1]{Gutierrez}:
\[ \partial v_\epsilon(E) \subset \partial u(E). \]

Now we can compute
\begin{align*}
\int_{\partial u(E)} (1+\abs{p}^2)^{-(n+2)/2}\,dp & \geq \int_{\partial v_\epsilon(E)}(1+\abs{p}^2)^{-(n+2)/2}\,dp\\
  &= \int_E \kappa_\epsilon(x)\,dx\\
	&= \int_E \kappa(x)\,dx + \int_{E\backslash E_\epsilon}(\kappa_\epsilon(x)-\kappa(x))\,dx\\
	&\geq \int_E \kappa(x)\,dx - 2\abs{E\backslash E_\epsilon}\sup\limits_{E}\kappa.
\end{align*}
Since $\kappa$ is bounded, we can take $\epsilon\to0$ to obtain
\bq\label{eq:superProof} \int_{\partial u(E)} (1+\abs{p}^2)^{-(n+2)/2}\,dp \geq \int_E \kappa(x)\,dx.\eq

{\bf Sub-solution.} Secondly, we construct a sub-solution by letting $w$ be the convex generalised solution of
\bq\label{eq:PDESub}
\begin{cases}
-\det(D^2w(x)) + \kappa_\epsilon(x)(1+\abs{\nabla w(x)}^2)^{(n+2)/2}, & x\in E\\
w(x) = u(x), & x\in\partial E.
\end{cases}
\eq
From Lemma~\ref{lem:aleksToVisc}, $w$ is a viscosity solution (and therefore sub-solution) of~\eqref{eq:PDESub} in $E$.
Since $w$ satisfies the Dirichlet boundary condition in the sense of Definition~\ref{def:weakDirichlet}, we must have $w \leq u$ on $\partial E$.  We can again apply the classical comparison principle to conclude that $w \leq u$ on $\bar{E}$.

Since $w$ satisfies the boundary condition in the generalised sense, we have that at each point $x\in\partial E$ either $w_*(x) = u(x)$ or the subgradient $\partial w_*(x)$ is empty; see Lemma~3.6 and Remark~(iii) of~\cite{Urbas_MA}.  Thus we can use Lemma~\ref{lem:subgradOrder} to order the subgradients:
\[ \partial u(E) \subset \partial w(E). \]
From here, we can easily compute
\bq\label{eq:subProof}
\int_{\partial u(E)}(1+\abs{p}^2)^{-(n+2)/2}\,dp \leq \int_{\partial w(E)}(1+\abs{p}^2)^{-(n+2)/2}\,dp = \int_E \kappa(x)\,dx.
\eq

Combining Equations~\eqref{eq:superProof} and~\eqref{eq:subProof}, we obtain
\[ \int_{\partial u(E)}(1+\abs{p}^2)^{-(n+2)/2}\,dp = \int_E \kappa(x)\,dx \]
for any uniformly convex set $E\subset\Omega$.

If $E\subset\Omega$ is non-uniformly convex, we can find a sequence of uniformly convex sets $E_j\subset E$ such that $\abs{E \backslash E_j} \to 0$.  Since $\kappa$ is bounded, we obtain
\[ \int_{\partial u(E)}(1+\abs{p}^2)^{-(n+2)/2}\,dp \geq \int_{E_j} \kappa(x)\,dx \to \int_E\kappa(x)\,dx.\]
The reverse inequality is proved similarily.

Finally, we can choose a generic $F\subset\Omega$.  For $\epsilon>0$, there exist non-overlapping convex sets $F_j$ such that $\bigcup\limits_j F_j \subset F$ and $\abs{F\backslash \bigcup\limits_j F_j} < \epsilon$.  Then the measure satisfies
\[ \int_{\partial u(F)}(1+\abs{p}^2)^{-(n+2)/2}\,dp \geq \sum\limits_j\int_{F_j}\kappa(x)\,dx \geq \int_F \kappa(x)\,dx - \epsilon\sup\limits_\Omega \kappa. \]
The reverse inequality is similar.  Then taking $\epsilon\to0$, we obtain
\[ \int_{\partial u(F)}(1+\abs{p}^2)^{-(n+2)/2}\,dp = \int_F \kappa(x)\,dx \]
and $u$ is a generalised solution.
\end{proof}

\section{Perron's Method}\label{app:perron}

In this appendix, we expand on the proof of Perron's method (Theorem~\ref{thm:perron}), expanding on the results of~\cite[Theorem~4.1]{BardiMannucci}.

\noindent {\bf Theorem 4} (Perron construction of viscosity solution). \emph{
Assume that $\Omega$, $g$, and $\kappa$ satisfy Hypothesis~\ref{hyp}.  If $u_1$ is an upper semi-continuous sub-solution and $u_2$ a lower semi-continuous super-solution with $u_1 \leq u_2$ on $\bar{\Omega}$ then
\[ w(x) = \sup\{W:\Omega\to\R \mid u_1 \leq W \leq u_2, \, W^* \text{ is a sub-solution}\} \]
is a viscosity solution of~\eqref{eq:MA}.}

\begin{proof}[Proof of Theorem~\ref{thm:perron}]
We notice that any admissible function $W$ appearing in the above definition is defined only in the interior $\Omega$.  Since $W^*$ is a sub-solution, it is convex (Lemma~\ref{lem:subConvex}), and therefore $W$ is continuous in $\Omega$ and satisfies $W^* = (W_*)^*$ in $\bar{\Omega}$.  By Lemma~\ref{lem:subBC}, $W^*$ is a sub-solution in the conventional sense (i.e. $W^* \leq g$ on $\partial\Omega$).  Thus we can use standard arguments to show that $w^*$ is a sub-solution and $w^*\leq g$ on $\partial\Omega$.  From Lemma~\ref{lem:subConvex}, $w^*$ is convex and $w^*=w$ in $\Omega$.

It remains to show that $w_*$ is a super-solution.  This is a standard argument for a classical Dirichlet problem, which requires super-solutions to satisfy $w^* \geq g$ on $\partial\Omega$, but is non-standard in our setting because we allow super-solutions to lie below the given Dirichlet data.  We verify the super-solution condition at every $x_0\in\bar{\Omega}$ and consider three possibilities.

{\bf Case 1}: $w_*(x_0) = u_2(x_0)$.  Choose any $\phi\in C^2$ such that $D^2\phi(x_0)>0$ and $w_*-\phi$ has a local minimum at $x_0$.  In this case
\[ u_2(x)-\phi(x) \geq w_*(x)-\phi(x) \geq w_*(x_0)-\phi(x_0) = u_2(x_0)-\phi(x_0)\]
and $u_2-\phi$ also has a local minimum at $x_0$.  Then since $u_2$ is a super-solution,
\[ F^*(x_0,w_*(x_0),\nabla\phi(x_0),D^2\phi(x_0)) \geq 0 \]
as required.

{\bf Case 2}: Interior points $x_0\in\Omega$ with $w_*(x_0) < u_2(x_0)$.  Suppose the super-solution condition is violated.  Then there exists some $\phi\in C^2$ such that $D^2\phi(x_0)>0$, $\phi(x_0)=w_*(x_0)$, $\phi(x)\leq w_*(x)$ nearby, and 
\[ -{\det}^+(D^2\phi(x_0)) + \kappa(x_0)R(\nabla\phi(x_0)) < 0. \]
We will derive a contradiction by constructing a sub-solution $v^*$ lying between $u_1$ and $u_2$ such that $v(x_0)>w(x_0)$, which contradicts the maximality of $w$. 

For sufficiently small $r,\epsilon,\gamma>0$ define
\[ v(x) = \begin{cases}
\max\{\phi(x) + \epsilon-\frac{\gamma}{2}\abs{x-x_0}^2,w(x)\} & \abs{x-x_0} < r\\
w(x), & \abs{x-x_0} \geq r.
\end{cases}\]
As long as the parameters are small enough, $v$ will satisfy the following conditions.
\begin{itemize}
\item $v \geq w \geq u_1$.  This follows trivially from the definition of $v$.
\item $v \leq u_2$. 
Let $\delta = u_2(x_0)-\phi(x_0) = u_2(x_0)-w_*(x_0)>0$.  Since $u_2$ is lower semi-continuous and $\phi$ is smooth, taking $r$ sufficiently small ensures that
\[ u_2(x) \geq u_2(x_0)-\frac{\delta}{3}, \quad \phi(x) \leq \phi(x_0) + \frac{\delta}{3} \]
whenever $\abs{x-x_0} < r$.  This in turn ensures that for sufficiently small $r,\epsilon$,
\[ u_2(x) \geq \phi(x) + \frac{\delta}{3} \geq  \phi(x) + \epsilon-\frac{\gamma}{2}\abs{x-x_0}^2 \]
and therefore $u_2(x)\geq v(x)$.
\item $-{\det}^+(D^2\phi(x)-\gamma I) + \kappa(x)R(\nabla\phi(x)-\gamma x) < 0$ for $\abs{x-x_0} < r$.  This is guaranteed for small enough $\gamma$ and $r$ since $\phi\in C^2$ and 
\[ -{\det}^+(D^2\phi(x_0)) + \kappa(x_0)R(\nabla\phi(x_0)) < 0. \]
\item $v$ is convex.  Since $\phi$ is smooth and uniformly convex in a neighbourhood of $x_0$, $v$ is the maximum of two convex functions and is therefore convex.
\end{itemize}

Now we verify that $v^*$ is a sub-solution.  As long as $r$ is sufficiently small, it is only necessary to check the conditions in $\Omega$, where $v^*=v$ and $w^*=w$.  Choose any $z\in{\Omega}$ and $\psi\in C^2$  such that $v-\psi$ has local maximum at $z$.  There are two possibilities depending on which function is active in the definition of $v$.

\emph{Case 2a}: $v(z) = w(z)$.  In this case
\[ w(z)-\psi(z) = v(z)-\psi(z) \geq v(x)-\psi(x) \geq w(x)-\psi(x).   \]
Since $w-\psi$ has a maximum at $z$ and $w$ is a sub-solution, the required condition is satisfied.

\emph{Case 2b}: $v(z) > w(z)$.  This means that nearby, $v(x) = \phi(x) + \epsilon-\frac{\gamma}{2}\abs{x}^2 \in C^2$.  Furthermore, this is only possible for $\abs{z-x_0} < r$.  Then since $v-\psi$ has a maximum at $z$,
\[ \nabla \phi(z) - \gamma x = \nabla\psi(z), \quad D^2\phi(z) - \gamma I \leq D^2\psi(z). \]
Consequently,
\begin{align*} F(z,v(z),\nabla\psi(z),D^2\psi(z)) &\leq -{\det}^+(D^2\psi(z)) + \kappa(z)R(\nabla\psi(z)) \\
   & \leq -{\det}^+(D^2\phi(z)-\gamma I) + \kappa(z)R(\nabla\phi(z)-\gamma z)\\
	 & < 0.
\end{align*}

This demonstrates that $v$ is a sub-solution satisfying $u_1 \leq v \leq u_2$ and $v(x_0) > w(x_0)$, which contradicts the maximality of $w$.  Therefore the super-solution condition must be satisfied at interior points.

{\bf Case 3}: Boundary points $x_0\in\partial\Omega$ with $w_*(x_0) < u_2(x_0)$.  We suppose again that $w_*$ violates the super-solution condition at $x_0$, which is only possible if $w_*(x_0) < g(x_0)$ (Lemma~\ref{lem:superBC}).    Then we can perform the same construction as in the previous case to generate a new function $v$:
\[ v(x) = \begin{cases}
\max\{\phi(x) + \epsilon-\frac{\gamma}{2}\abs{x-x_0}^2,w_*(x)\}, & \abs{x-x_0} < r\\
w_*(x), & \abs{x-x_0} \geq r.
\end{cases}\]
As above, $v^*$ is a sub-solution  in $\Omega$.  

Now we consider $z\in\partial\Omega$.
Recall that $\phi(x_0) = w_*(x_0)<g(x_0)$.  Since $\phi$ and $g$ are continuous, sufficiently small $\epsilon$ and $r$ ensure that
\[ \phi(z) + \epsilon-\frac{\gamma}{2}\abs{z-x_0}^2 < g(z), \quad \abs{z-x_0}<r. \]
Thus 
\[ v^*(z) \leq \max\left\{\phi(z) + \epsilon-\frac{\gamma}{2}\abs{z-x_0}^2,w^*(z)\right\} \leq g(z) \]
 and $v^*$ is a sub-solution.

Finally, we observe that for sufficiently small $\abs{x-x_0}$, $v(x) \geq w_*(x_0)+\epsilon/2$ and therefore $v_*(x_0)>w_*(x_0)$.  This, in turn, requires that $v(x)>w(x)$ for some $x\in\Omega$, which contradicts the maximality of~$w$.

We conclude that the function $w_*$ generated by the Perron construction must a super-solution and therefore $w$ is a viscosity solution of~\eqref{eq:MA}.
\end{proof}

}

\end{document}